\documentclass[aihp]{imsart}

\RequirePackage{amsthm,amsmath,amsfonts,amssymb}
\RequirePackage[numbers,sort&compress]{natbib}
\RequirePackage[colorlinks=true,urlcolor=magenta!60!black,citecolor=green!50!black,linkcolor=red!90!black,bookmarksnumbered=true,pdfauthor={Bettinelli,Curien,Fredes,Sepulveda}]{hyperref}
\RequirePackage{graphicx}
\usepackage{tikz}
\usepackage{subcaption}
\usepackage{enumitem}
\usepackage{todonotes}
\usepackage{array,xcolor}
\usepackage{stmaryrd}\SetSymbolFont{stmry}{bold}{U}{stmry}{m}{n}
\usepackage{cleveref}

\startlocaldefs
\numberwithin{equation}{section}

\theoremstyle{plain}

\newtheorem{thm}{Theorem}[section]
\newtheorem{lemma}[thm]{Lemma}
\newtheorem{prop}[thm]{Proposition}

\theoremstyle{remark}

\newtheorem{rem}[thm]{Remark}

\newcommand{\GQ}{\mathcal{Q}}
\newcommand{\SQ}{\widetilde{\mathcal{Q}}}
\newcommand{\tQ}{\widetilde{Q}}

\newcommand{\Sq}{\tilde{q}}

\newcommand{\dd}{\mathrm{d}}
\DeclareMathOperator{\dis}{dis}
\newcommand{\dGH}{\mathrm{d_{GH}}}
\newcommand{\dTV}{\mathrm{d_{TV}}}

\newcommand{\br}{\mathbf{r}}
\newcommand{\bq}{\mathbf{q}}

\newcommand{\mm}{\mathbf{m}}
\renewcommand{\tt}{\mathbf{t}}

\newcommand{\Z}{\mathbb Z}

\newcommand{\N}{\mathbb N}

\newcommand{\cJ}{\mathcal{I}}

\newcommand{\cR}{\mathcal{R}}
\newcommand{\cS}{\mathcal{S}}
\newcommand{\fS}{\mathfrak{S}}
\newcommand{\cT}{\mathcal{T}}
\newcommand{\cX}{\mathcal{X}}
\newcommand{\cY}{\mathcal{Y}}
\newcommand{\cZ}{\mathcal{Z}}
\newcommand{\fB}{\mathfrak{B}}
\newcommand{\fF}{\mathfrak{F}}
\newcommand{\fL}{\mathfrak{L}}
\newcommand{\fW}{\mathfrak{W}}

\newcommand{\E}{\mathbb E}

\newcommand{\eps}{\varepsilon}

\DeclareMathOperator{\core}{Core}
\DeclareMathOperator{\Rt}{Rt}

\newcommand{\RR}{\mathfrak R}

\renewcommand{\P}{\mathbb P}
\newcommand{\ro}{\cR^{\eps}_{n}}
\newcommand{\rt}{\cR'^{\eps}_{n}}
\newcommand{\cro}{\bar{\cR}^{\eps}_{n}}
\newcommand{\crt}{\bar{\cR}'^{\eps}_{n}}
\newcommand{\de}{\mathrel{\mathop:}\hspace*{-.6pt}=}
\DeclareMathOperator*{\lsup}{\smash{\limsup}}
\DeclareMathAlphabet\mathbfcal{OMS}{cmsy}{b}{n}

\newcommand{\BD}{\mathbf{BD}}
\newcommand{\are}[1]{\lVert #1\rVert}
\newcommand{\arelr}[1]{\left\lVert #1\right\rVert}
\DeclareMathOperator{\suc}{succ}
\DeclareMathOperator{\argmin}{argmin}
\newcommand{\per}[1]{\lvert\partial #1\rvert}
\newcommand{\perlr}[1]{\left\lvert\partial #1\right\rvert}
\newcommand{\ind}{\mathbf 1}
\newcommand{\pin}{p_\mathrm{in}}
\newcommand{\pri}{p_\mathrm{left}}
\newcommand{\ple}{p_\mathrm{right}}
\newcommand{\bge}{{\scalebox{.7}{$\boldsymbol\ge$}}}
\newcommand{\beq}{{\scalebox{.7}{$\boldsymbol=$}}}
\newcommand{\ouM}{\overline{\underline{\mathrm{M}}}}
\endlocaldefs

\begin{document}

\begin{frontmatter}
\title{Scaling limit of random plane quadrangulations with a simple boundary, via restriction}
\runtitle{Scaling limit of random plane quadrangulations with a simple boundary, via restriction}

\begin{aug}
\author[A]{\inits{J.}\fnms{J\'er\'emie}~\snm{Bettinelli}\orcid{	0000-0002-0359-7493}},
\author[B]{\inits{N.}\fnms{Nicolas}~\snm{Curien}},
\author[C]{\inits{L.}\fnms{Luis}~\snm{Fredes}\orcid{0000-0003-1404-6328}}
\and
\author[D]{\inits{A.}\fnms{Avelio}~\snm{Sep\'ulveda}\orcid{0000-0001-9481-8898}}

\address[A]{LIX, cnrs, \'Ecole polytechnique, Institut Polytechnique de Paris, Palaiseau, France.}
\address[B]{Universit\'e Paris-Saclay and Institut universitaire de France, Orsay, France}
\address[C]{Universit\'e de  Bordeaux, cnrs, Bordeaux INP, IMB, Talence, France}
\address[D]{Universidad de Chile, Centro de Modelamiento Matem\'atico (AFB170001), UMI-CNRS 2807, Beauchef 851, Santiago, Chile.}
\end{aug}


\begin{abstract}
We prove that quadrangulations with a simple boundary converge to the Brownian disk. More precisely, we fix a sequence~$(p_n)$ of even positive integers with $p_n\sim 2\alpha \sqrt{2n}$ for some $\alpha\in(0,\infty)$. Then, for the Gromov--Hausdorff topology, a quadrangulation with a simple boundary uniformly sampled among those with~$n$ inner faces and boundary length~$p_n$ weakly converges, in the usual scaling $n^{-1/4}$, toward the Brownian disk of perimeter~$3\alpha$. 

Our method consists in seeing a uniform quadrangulation with a simple boundary as a conditioned version of a model of maps for which the Gromov--Hausdorff scaling limit is known. We then explain how classical techniques of unconditionning can be used in this setting of random maps.
\end{abstract}

\begin{abstract}[language=french]
Nous prouvons que les quadrangulations \`a bord simple convergent vers le disque brownien. Plus pr\'ecis\'ement, nous fixons une suite~$(p_n)$ d'entiers pairs strictement positifs tels que $p_n\sim 2\alpha \sqrt{2n}$ pour un certain $\alpha\in(0,\infty)$. Alors, pour la topologie de Gromov--Hausdorff, une quadrangulation \`a bord simple, choisie uniform\'ement au hasard parmi celles ayant~$n$ faces internes et p\'erim\`etre~$p_n$, converge faiblement, dans l'\'echelle usuelle $n^{-1/4}$, vers le disque brownien de p\'erim\`etre~$3\alpha$.

Notre m\'ethode consiste \`a consid\'erer une quadrangulation \`a bord simple uniforme comme une version conditionn\'ee d'un mod\`ele de cartes pour lequel la limite d'\'echelle au sens de Gromov--Hausdorff est d\'ej\`a connue. Nous expliquons ensuite comment utiliser les techniques classiques de d\'econditionnement dans ce contexte de cartes al\'eatoires. 
\end{abstract}

\begin{keyword}[class=MSC]
\kwd[primary ]{60F99}
\kwd{60D05}
\kwd[; secondary ]{05C80}
\end{keyword}

\begin{keyword}
\kwd{plane maps}
\kwd{Brownian disk}
\kwd{quadrangulation}
\kwd{scaling limit}
\kwd{simple boundary}
\end{keyword}

\end{frontmatter}

\section{Introduction}

In probability theory, proving conditional limit theorems is usually much harder than obtaining the corresponding unconditional versions; for instance, one may think of conditional versions of Donsker's theorem (e.g.~\cite{kaigh76ipr}). In the present work, we describe a method enabling to transfer the convergence of some model of random maps to a similar model with extra constrains (here obtained by imposing simplicity conditions on the boundary). This is inspired from well-known techniques used for random processes or random trees, see e.g.~\cite{legall10exc,kortch13duq,DI77}.

\subsection*{Plane maps} A \emph{plane map} is an embedding of a finite connected graph (possibly with loops and multiple edges) into the two-dimensional sphere, considered up to direct homeomorphisms of the sphere. The faces of the map are the connected components of the complement of the union of the edge set. We will particularly focus on \emph{quadrangulations with a boundary}, which are particular instances of plane maps whose faces are all \emph{quadrangles}, that is, of degree~$4$, with the exception of one face of arbitrary even degree. The latter face will be referred to as the \textit{external face}, whereas all others will be called \textit{inner faces}; the number of inner faces is the \emph{area} of the map. We say that an oriented edge, that is, an edge given with one of its two possible orientations, is \emph{incident} to a face if it lies on boundary of the face, with the face on its right\footnote{In the literature, it is also common to use the convention that the face lies to the left. The present convention will make the encoding of Section~\ref{secbij} easier.}. The oriented edges incident to the external face will constitute the \textit{boundary} of the map and the degree of the external face is called the \emph{length} of the boundary or the \emph{perimeter} of the map. In general, we do not require the boundary to be a simple curve; when it is, we speak of \emph{quadrangulations with a simple boundary}. Unless explicitly stated, we will always consider our maps to be rooted, which means that one of the oriented edges, called the \emph{root} of the map, is distinguished. In the case of quadrangulations with a boundary, the root will always be incident to the external face, that is, lie on the boundary, with the external face to its right. See Figure~\ref{core} for an example of quadrangulations with either a general or a simple boundary. For $n\in\N$ and $p\in 2\N$, we denote by $\GQ_{n,p}$ the set of quadrangulations with a boundary\footnote{Beware that, in the present work, the second index is always even and represents the perimeter of the map. In the literature, it is common to use half the perimeter instead. As the boundary of maps considered here will be broken into pieces of arbitrary parity, we found this convention more appropriate.} having~$n$ inner faces and perimeter~$p$, as well as $\SQ_{n,p}\subseteq \GQ_{n,p}$ the subset of quadrangulations with a \emph{simple} boundary. By convention, we see the map with one edge and two vertices as the only element of $\SQ_{0,2}=\GQ_{0,2}$.	For~$\bq\in\GQ_{n,p}$, we respectively denote its area and perimeter by
\[
\are{\bq}\de n\qquad\qquad\text{ and }\qquad\qquad\per{\bq}\de p\,.
\]

\begin{figure}[ht!]
	\centering\includegraphics[height=51mm]{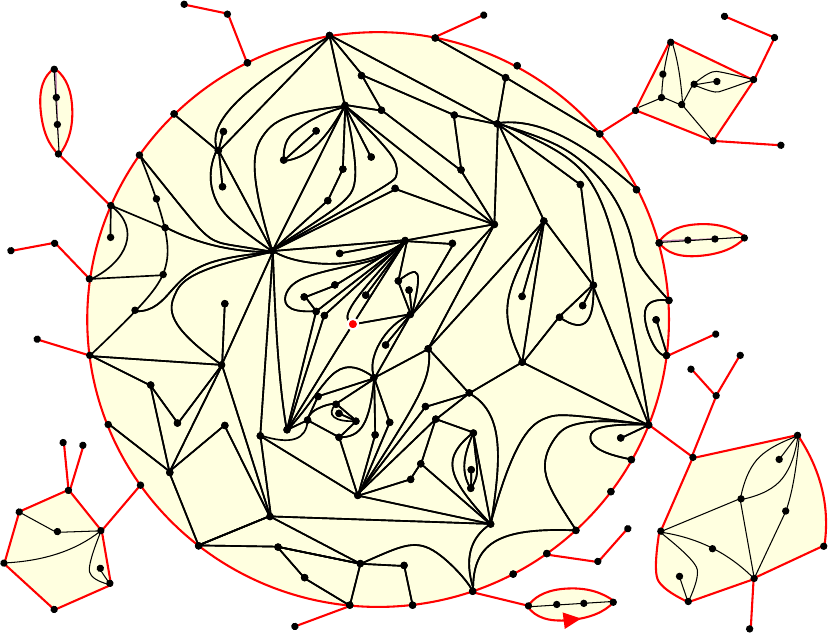}\hspace{2cm}\includegraphics[height=51mm]{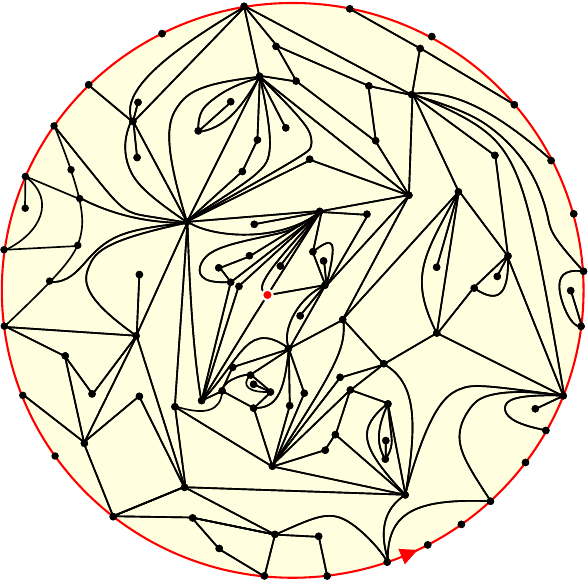}
	\caption{Quadrangulation with a boundary on the left; quadrangulation with a \emph{simple} boundary on the right. The boundary is represented in red. These maps are \emph{pointed} in the sense that a vertex, in red, is distinguished. The map on the right is in fact the so-called \emph{core} of the pointed map on the left, defined in Section~\ref{seccoredec}.}
	\label{core}
\end{figure}

For technical reasons due to bijective encodings, we will often consider pointed maps: we say that a map~$\mm$ is \emph{pointed} if it is given with a distinguished element of its vertex-set $V(\mm)$. We introduce the sets of quadrangulations with a simple boundary and that of pointed quadrangulations with a simple boundary: 
\[
\tQ\de\SQ_{0,2}\cup\bigcup_{n\in\N,\,p\in2\N} \SQ_{n,p}\qquad\text{ and }\qquad
\tQ^\bullet\de\Big\{(\bq,\rho)\,:\,\bq\in\tQ,\, \rho\in V(\bq)\Big\}\,.
\]

\subsection*{Proving convergence toward the Brownian sphere} The \emph{Brownian sphere}~\cite{LG11,miermont11bms} is a random fractal metric space almost surely homeomorphic to the sphere that appears as a universal scaling limit of many models or random plane maps. In his breakthrough work \cite{LG11}, Le Gall gave a robust path to prove the convergence of a family of random maps toward the Brownian sphere; it has since been used in many works~\cite{beltranlg,addarioalbenque2013simple,BeJaMi14,abr13,AHS19simpletrigpolygons,AlbenqueAddarioOdd}. One downside of this method is that it requires to find a bijective encoding ``\`a la Schaeffer'' of the family of plane maps in question by a suitable class of labeled trees. A different approach has been taken in \cite{CLGmodif} where it is shown that ``local modifications'' of distances in uniform triangulations only change the large scale metric by a multiplicative factor (which is unknown in most cases). This has later been extended to the case of Eulerian triangulations~\cite{carrance2019convergence} and quadrangulations~\cite{lehericy2019firstpassage}. Another direct method is to transfer results to classes of maps that are ``contained within'' another class, for instance by taking the core decomposition, by pruning the boundary, etc. This usually yields a family of random maps $M_{N(n)}$, which converges in the scaling limit but for which the ``size'' $N(n)$ (which may be the number of faces, the length of the boundary, etc.) is random and satisfies a weak law of large number $ N(n)/n \to c$ for some $c>0$. Examples of such constructions can be found in~\cite{BFSS01,CM12,ABWcore,GM19cv}. It then remains to deduce from such results the convergence of~$M_{n}$ as $n \to \infty$ by unconditioning methods. We will use in this work such a method: the idea is to consider restrictions of our map model obtained by ``exploring'' all but a tiny proportion of the map. The law of these restrictions are then controlled in total variation distance using a ``local limit theorem'' (here exact counting of maps). The remaining of the argument consists in establishing that those restrictions are close to the whole map.

\subsection*{Setting and notation.} For each $n\in\N$ and $p\in 2\N$, we let~$Q_{n,p}$ be uniformly distributed over the set~$\GQ_{n,p}$ of quadrangulations with~$n$ inner quadrangles and a general boundary of length~$p$, as well as~$\tQ_{n,p}$ be uniformly distributed over the set~$\SQ_{n,p}$ of quadrangulations with~$n$ inner quadrangles and a simple boundary of length~$p$. We also denote by~$Q^{\bullet}_{n,p}$ and~$\tQ_{n,p}^\bullet$ uniform quadrangulations respectively of~$\GQ_{n,p}$ and~$\SQ_{n,p}$ that are pointed uniformly at random on one of their vertices.

When~$\mm$ is a map, we equip its vertex-set~$V(\mm)$ with the graph metric~$\dd_\mm$ defined as the minimal number of edges in a path linking vertices. Furthermore, for a positive number $c >0$, we denote by $c\, \mm$ the (finite) metric space $(V(\mm), c \,\dd_\mm (\cdot,\cdot))$; a map or a pointed map may thus be seen as a metric space.

From now on, we fix $\alpha\in(0,\infty)$ and a sequence~$(p_n)_{n\in\N}$ with
\[
p_n\sim 2\alpha \sqrt{2n}\qquad\text{ as }n\to\infty\,.
\] 

\subsection*{Scaling limit of quadrangulations with a boundary.} 
In the present work, we show the convergence of quadrangulations with a simple boundary toward the Brownian disk. This particular choice of random maps model is motivated by the study of gluing operations on maps \cite{CC16,GM2,GP19,FS19}. The \emph{Brownian disk}~\cite{BeMi15I} is the counterpart to the Brownian sphere with the topology of the disk. It arises as the scaling limit of many models of random plane maps with a boundary (that is, plane maps with only one large face in the scale~$\sqrt n$, where~$n$ is the number of faces). In particular, the following convergence is established~(\cite[Theorem~1]{BeMi15I}):
\begin{equation}\label{cvBdisk}
\left( \frac{9}{8n}\right)^{1/4}Q_{n, p_n} \xrightarrow[n\to\infty]{(d)} \BD_{\alpha},
\end{equation}
in distribution for the Gromov--Hausdorff topology\footnote{See Appendix~\ref{apdGH}.}, where $\BD_{\alpha}$ is the Brownian disk with perimeter~$\alpha$ and unit area. Using the conveniences of \cite{BeMi15I}, as well as the peeling process, Gwynne \& Miller~\cite{GM19cv} later proved that properly rescaled quadrangulations with a \emph{simple} boundary but with random area (under the critical Boltzmann distribution) converge toward a free area version of the Brownian disk called the \emph{free Brownian disk}~\cite[Section~1.5]{BeMi15I}. We prove the following conditional version of this convergence.

\begin{thm}\label{mainthm} It holds that
\[
\left( \frac{9}{8n}\right)^{1/4} \tQ_{n,p_{n}}\xrightarrow[n \to \infty]{(d)} 
\BD_{3 \alpha},
\]
in distribution for the Gromov--Hausdorff topology.
\end{thm}
One might be surprised to obtain the same scaling limit (up to a constant) as for maps with a general boundary~\eqref{cvBdisk} but, in fact, it was known that the boundary is ``simple at the limit,'' in the sense that the Brownian disk is homeomorphic to a disk~\cite{BetQ}. In this regard, it was expected to obtain the same limit, only with a different boundary length. This boundary factor will appear clearly in a moment.

\begin{rem}
In fact, the convergence of \Cref{mainthm} can be strengthen to the more elaborate Gromov--Hausdorff--Prohorov--Uniform topology \cite[Section~1.2.3]{GM19cv}, which furthermore keeps track of the area and perimeter measures on the map. We chose to use the present simpler framework as we believe the latter would make the paper harder to read and longer, and lead us farther away from the method we chose to present here.
\end{rem}

The remainder of the paper is organized as follows: in the next section, we prove the above theorem assuming technical propositions. As we said above, the idea is to use a proxy for $\tQ_{n,p_{n}}$ for which we know the convergence to the Brownian disk, and then to establish ``local absolute continuity relations.'' In our case, the proxy will be the so-called \emph{core} of a (general) random quadrangulation, and the local absolute continuity relations will be obtained by considering appropriate restrictions of those maps. The proofs of the technical propositions are then derived in Sections~\ref{seccomp} and~\ref{secbij} using exact counting and the usual bijective construction for the proxy model.

\section{Method of proof}

In this section, we present the main lines of the proof of \Cref{mainthm}, deferring the technical estimates to the next sections. This choice of presentation is motivated by the fact that the overall scheme is somehow disconnected from the technical estimates and might be adapted to other similar situations, at the price of appropriate estimates.

\subsection{Core decomposition and proxy map}\label{seccoredec}

Fix a pointed quadrangulation~$\bq^{\bullet}= (\bq,\rho)$ with a general boundary. Its \emph{core}, denoted by $\core(\bq^{\bullet})$, is the pointed quadrangulation with a simple boundary defined as follows; see Figure~\ref{core}. By ``cutting'' the pinch vertices along its boundary, we may decompose~$\bq$ into smaller quadrangulations with a simple boundary, each rooted at the first oriented edge of its boundary in contour order starting from the root of~$\bq$. If there is a unique largest such component (in terms of number of inner faces) and~$\rho$ belongs to this component, then the core is the latter component. Otherwise, we define $\core(\bq^{\bullet})$ as an abstract cemetery point~$\wp$ for which we set $\are{\wp}=\per{\wp}\de 0$. Let us first remind a few well-known properties of the core of a random quadrangulation with a general boundary; see~\cite[Section~4]{CM12} for more information.  

\begin{prop}[{\cite[Proposition 2.6 \& Lemma 2.7]{GM19cv}}]\label{cvAnPn}
We have $\P\big(\core(Q^{\bullet}_{n,3p_{n}}) \neq \wp\big) \to 1$ as $n \to \infty$ and, furthermore,
\[
\frac{\arelr{\core\big(Q^{\bullet}_{n,3p_{n}}\big)}}{n}  \xrightarrow[n\to\infty]{(\P)}1\qquad\text{ and }\qquad
	\frac{\perlr{\core\big(Q^{\bullet}_{n,3p_{n}}\big)}}{p_n} \xrightarrow[n\to\infty]{(\P)} 1\,.
\]
\end{prop}
Furthermore, conditionally given the area $\tilde A_{n} \de \are{\core(Q^{\bullet}_{n,3p_{n}})}$ and perimeter $\tilde P_{n} \de \per{\core(Q^{\bullet}_{n,3p_{n}}) }$, provided that $\tilde A_{n} > n/2$ to avoid possible ties,
\[\core\big(Q^{\bullet}_{n,3p_{n}}\big) \quad  \text{ is uniformly distributed over }  \quad \tQ^{\bullet}_{\tilde A_{n},\tilde P_{n}}\,.\]
In particular, the core of $Q_{n,3p_{n}}^{\bullet}$ contains most of the map and indeed~\eqref{cvBdisk} can be strengthened into 
\begin{equation}\label{cvCore}
\left(\left(\frac{9}{8n}\right)^{1/4}Q_{n, 3p_n}^{\bullet},\left( \frac{9}{8n}\right)^{1/4}\core\big(Q_{n, 3p_n}^{\bullet}\big)\right) \xrightarrow[n\to\infty]{(d)} \Big(\BD_{3\alpha},\BD_{3\alpha}\Big)\,.
\end{equation}
The above joint convergence is obtained in~\cite[Theorem~1.3]{GM19cv}, together with the addition of a natural parameterization of the boundary. Combining the above remarks, we might seem close to our goal since $\SQ_{n, p_{n}}^{\bullet}\approx\SQ_{\tilde A_n,\tilde P_n}$, which has the same distribution as $\core\big(Q_{n, 3p_n}^{\bullet}\big)$; in particular this explains the boundary factor~$3$ in \Cref{mainthm}. It remains to lift the previous convergence to a conditional convergence when the area and perimeter are fixed. To do this, we will prove that the distributions of ``large parts'' of $\core\big(Q_{n, 3p_n}^{\bullet}\big)$ and of $\tQ_{n,p_{n}}^{\bullet}$ may be rendered arbitrarily close in total variation distance. These large parts will be defined via what we call \emph{restrictions}.

\subsection{Restrictions}\label{secrestr}

For each $\eps >0$ and $n \ge 1$, we will define \emph{restrictions} of $\tQ_{n,p_{n}}^{\bullet}$ and of $\core\big( Q_{n, 3p_{n}}^{\bullet} \big)$ obtained by exploring the maps in question  up to an $\eps$-small part. These unexplored parts will have a random number of inner faces and a random perimeter, and we will see in the next section, using \emph{exact counting} results, that the restrictions in both models are close in total variation distance.

Given a pointed map $(\bq,\rho)$ and $\ell\in \N$, we denote by $B_\ell(\bq,\rho)$ its \emph{ball} of radius~$\ell$, that is, the map obtained from~$\bq$ by keeping only the faces that are incident to at least one vertex lying at graph distance~$\ell-1$ or less from the marked vertex~$\rho$. 

The notion of restriction we will use roughly consists in taking the (hull of the) smallest ball that hits the boundary of the map within distance~$\eps\,p_{n}$ from the vertex of the boundary located roughly at a third of the boundary length from the root. The choice of taking a third comes from the need to have two ``well overlapping'' restrictions to apply a resampling argument in Section~\ref{sec:resampling}. In some cases, the construction will not work properly and the definition of the restriction in such a case will not matter too much since these cases should happen with negligible probability in the end.

We fix $n\in \N$ and $\eps >0$, and we define the restriction $\ro$ and its ``complement'' $\cro$ as follows; see Figure~\ref{RBall}. Let $(\bq,\rho)$ be a pointed quadrangulation with a simple boundary and denote by~$p$ its perimeter.\footnote{Note that the area of~$\bq$ is not specified; in practice, this construction will be applied when the area is roughly~$n$.} We assume that $p \ge p_{n}/2$ and number the vertices of the boundary of~$\bq$ from~$0$ to~$p-1$ starting from the tail of the root and following the orientation given by the root. We furthermore assume that $\eps<1/3$ and define~$I$ as the set of vertices of the boundary of~$\bq$ that are numbered from~$\big\lfloor (\frac{1}3-\eps)\, p_{n}\big\rfloor$ to~$\lfloor p_{n}/3\rfloor$, the latter vertex being denoted by $t_{1/3}$ and thought of as ``the target vertex located at a third of the way around the boundary.''

\begin{figure}[ht!]
	\centering\includegraphics[scale=0.55]{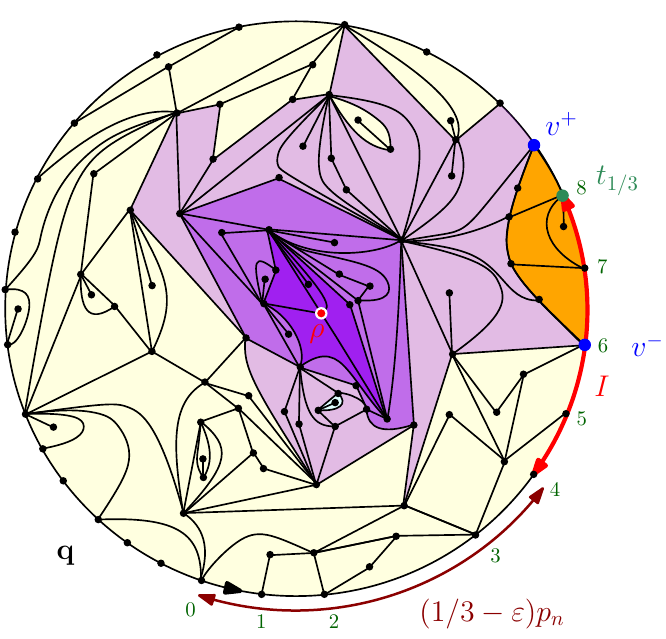}\hfill\includegraphics[scale=0.55]{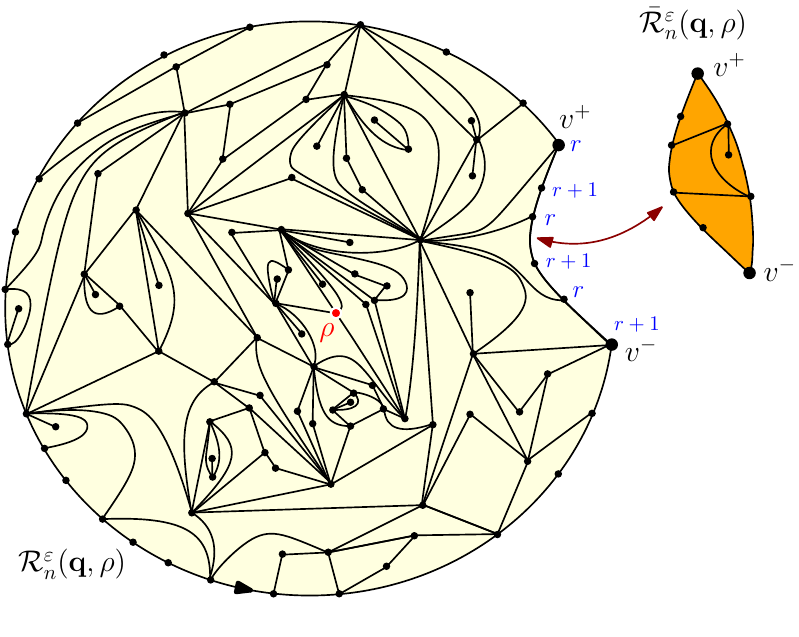}
	\caption[]{Definition of the restriction. We consider the smallest ball that hits the boundary of the map at boundary length between $(\frac{1}3-\eps)\, p_{n}$ and $p_{n}/3$ from the root. On this example, $p=26$, $\big\lfloor (\frac{1}3-\eps)\, p_{n}\big\rfloor=4$, $\lfloor p_{n}/3\rfloor=8$ and $r=3$. The balls of radius~$1$ to~$3$ are depicted with lighter and lighter shades of purple. The restriction $\ro(\bq, \rho)$ is the map consisting of this ball with the addition of the (light yellow) components that do not contain~$t_{1/3}$; the so-called \emph{complement}\protect\footnotemark{} $\cro{(\bq,\rho)}$ is the (orange) component that contains~$t_{1/3}$.}
	\label{RBall}
\end{figure}
\footnotetext{Beware that it is a complement in term of faces, not in term of edges and vertices because of the boundaries.}

We let~$r$ be the smallest integer such that the ball $B_{r}(\bq,\rho)$ intersects~$I$ and denote by~$v^-$ the last vertex of $I\cap B_{r}(\bq,\rho)$, that is, the vertex of this set whose number is the largest in the above numbering of the boundary vertices. We also assume that $B_{r}(\bq,\rho)$ hits the ``other side'' of the boundary between $t_{1/3}$ and the root (both excluded), and denote by~$v^+$ the vertex of $B_{r}(\bq,\rho)$ on the boundary of~$\bq$ with smallest number above $\lfloor p_{n}/3\rfloor + 1$. Notice that, depending on how the ball $B_{r}(\bq,\rho)$ ``hits'' the boundary of~$\bq$, the vertex~$v_{-}$ may be at distance~$r$ or $r+1$ from~$\rho$ and the same goes for~$v_{+}$. 

When the above conditions are satisfied, we set $\ro(\bq,\rho)$ to be the so-called hull of $B_{r}(\bq,\rho)$ with respect to $t_{1/3}$, roughly obtained by filling all the ``holes'' of~$\bq$ except the one containing~$t_{1/3}$. More precisely, it is defined as follows.
\begin{itemize}
	\item If all the faces incident to the part of the boundary of~$\bq$ from~$v^-$ to~$v^+$ belong to~$B_{r}(\bq,\rho)$, then we set $\ro(\bq,\rho)\de(\bq,\rho)$ and $\cro(\bq,\rho)$ as the map with one edge and two vertices.
	\item Otherwise, the inner faces of~$\bq$ that do not belong to~$B_{r}(\bq,\rho)$ are gathered into subsets of adjacent\footnote{Two faces are \emph{adjacent} if they are incident to the same edge. Note that two faces ``only touching by a vertex'' are not adjacent.} faces and only one of these subsets contains faces incident to the part of the boundary of~$\bq$ from~$v^-$ to~$v^+$; we denote this subset by~$\mathfrak{C}$. We define $\ro(\bq,\rho)$ as the map obtained from~$\bq$ by suppressing the faces of~$\mathfrak C$, as well as all the edges and vertices that are only incident to faces of~$\mathfrak C$. We also let $\cro(\bq,\rho)$ be the map obtained from~$\bq$ by keeping the faces of~$\mathfrak C$, as well as all the edges and vertices that are incident to those faces.
\end{itemize}

The map $\ro(\bq,\rho)$ is a quadrangulation with a simple boundary that contains the root edge, the pointed vertex~$\rho$ and with two additional distinguished points~$v^-$ and~$v^+$ on its boundary. Observe that the part of its boundary between~$v_{-}$ and~$v_{+}$ is made of vertices whose distances to the vertex~$\rho$ alternate between~$r$ and $r+1$. The map $\cro(\bq,\rho)$ is a nonrooted quadrangulation with a simple boundary with two distinguished points~$v^-$ and~$v^+$ on its boundary. In the case when the above construction cannot be performed, $\ro(\bq,\rho)$ and $\cro(\bq,\rho)$ are set to the abstract cemetery point~$\wp$.

\begin{rem} At this point, the reader might wonder why we do not use the root as basepoint for balls instead of a randomly chosen vertex~$\rho$. This is only to ease the proof of the forthcoming technical propositions because the bijective encoding of maps are easier to deal with when measuring distances from a random chosen vertex rather than from the root edge; see~\cite{LeGallBoundarypoint}. 
\end{rem}

Observe that $\ro(\bq,\rho)$ is ``decreasing'' with~$\eps$ in the sense that, for $0<\eta<\eps$, the map $\ro(\bq,\rho)$ is ``contained'' in~${\cR^{\eta}_{n}}(\bq,\rho)$. We leave this notion of submap at an intuitive level as we will not really need it in this work. We will only use the fact that, 
\begin{equation}\label{decreas}
\text{for $0<\eta<\eps$,}\qquad {\cR^{\eta}_{n}}(\bq,\rho)={\cR^{\eta}_{n}}(\bq',\rho')\quad\implies\quad \ro(\bq,\rho)=\ro(\bq',\rho')\,.
\end{equation}

Another important feature of this construction is that~$\ro(\bq,\rho)$ and~$\cro(\bq,\rho)$ are independent in the sense that any map~$\bq'$ obtained by completing the map~$\br=\ro(\bq,\rho)$ on the part of its boundary between~$v^-$ and~$v^+$ satisfies $\ro(\bq',\rho)=\br$. This is the reason why we defined the set~$I$ from ``within'' $\ro(\bq,\rho)$. We will come back to this in Section~\ref{seclawrest}.

\subsection{Proof of Theorem~\ref{mainthm} provided two technical estimates}

We now present the main lines of the proof of \Cref{mainthm}. Let us set
\begin{equation} \label{eq:notationref}
\underbrace{X_n\de \tQ_{n,p_{n}}^{\bullet}}_{\textit{\color{blue!80!black}model under study}}\,,\qquad \underbrace{Y_n\de \core\big(Q_{n, 3p_n}^{\bullet}\big)}_{ \textit{\color{blue!80!black}reference model}}\,,\qquad\text{ and } a_n\de\left(\frac{9}{8n}\right)^{1/4} \,.
\end{equation}
The classical bijective encodings often lack flexibility: for instance, tracking through the usual Schaeffer-like bijection\footnote{See Section~\ref{secbij}.} the condition that the boundary is simple is very intricate. In this paper, these bijective encodings will only be used in order to obtain (rough) estimates \emph{on the reference model}. For the model under study, our method only requires counting results.

First, the convergence of the second coordinate of~\eqref{cvCore} ensures that
\begin{equation}\label{cvyn}
a_nY_n \xrightarrow[n \to \infty]{(d)} \BD_{3 \alpha}\,,
\end{equation}
in distribution for the Gromov--Hausdorff topology. Our goal is to obtain a similar statement with~$X_n$ in place of~$Y_n$. This will follow from the facts that the distributions of~$\ro(X_n)$ and of~$\ro(Y_n)$ are close and the leftover parts~$a_n\cro(X_n)$ and~$a_n\cro(Y_n)$ are not too large (when~$\eps$ gets small). These conditions are gathered into the following propositions, whose proofs are postponed to the subsequent sections. In the following, we write $\dTV(A,B)$ for the total variation distance between the distributions of two random variables~$A$ and~$B$. The following proposition will be proved in Section~\ref{seccomp}.

\begin{prop}[Restrictions are close]\label{Pclose}
For all $\eps>0$,
\[\lim_{n \to \infty} \dTV\big(\ro(X_n),\ro(Y_n)  \big)=0\,.\]
\end{prop}

We denote by~$\dGH$ the Gromov--Hausdorff metric on isometry classes of metric spaces. The following proposition will be proved in Section~\ref{secbij}.

\begin{prop}[Leftover is small]\label{Psmall}
The following holds.
\begin{enumerate}[label=(\textit{\roman*})]
	\item For every $ \delta >0$, $\displaystyle\lim_{\eps \to 0}\lsup_{n\to\infty}\P\big( \dGH(a_n Y_n, a_n\ro(Y_n))>\delta\big)=0$\,.\label{PsmallY}
	\item For every $ \delta >0$, $\displaystyle\lim_{\eps \to 0}\lsup_{n\to\infty}\P\big( \dGH(a_n X_n, a_n\ro(X_n))>\delta\big)=0$\,.\label{PsmallX}
\end{enumerate}
\end{prop}

\begin{proof}[Proof of \Cref{mainthm}]
The result follows from a coupling argument. Thanks to Skorohod's embedding theorem, we may assume that we work on a probability space where the convergence~\eqref{cvyn} holds almost surely: let us denote by~$Y$ the limit. Let $f$ be a bounded uniformly continuous real-valued function on the set of isometry classes of compact metric spaces and~$\eta>0$. There exists~$\delta>0$ such that
\[\dGH(\cX,\cY)<3\delta \implies |f(\cX)-f(\cY)|<\eta\,.\]
Then
\begin{align}
\Big|\E\big[f(a_n X_n)-f(Y)\big]\Big|&\le \E\big[|f(a_n X_n)-f(Y)|\,\ind_{\{\dGH(Y,a_nX_n)< 3\delta\}}\big]+\E\big[|f(a_n X_n)-f(Y)|\,\ind_{\{\dGH(Y,a_nX_n)\ge 3\delta\}}\big]\notag\\
	&\le \eta+2\sup(|f|)\, \P\big(\dGH(Y,a_nX_n)\ge 3\delta\big)\,.\label{eqpfthm1}
\end{align}
We then write
\begin{align*}
\P\big(\dGH(Y,a_nX_n)\ge 3\delta\big)&\le \P\big(\dGH(Y,a_nY_n)\ge \delta \big) + \P\big(\dGH(a_nY_n,a_nX_n)\ge 2\delta\big).
\end{align*}
Due to the convergence $a_n Y_n\to Y$, the first term in the right-hand side tends to~$0$ as $n \to \infty$. The second term is bounded from above by
\[\P\big(\dGH(a_nY_n,a_n \ro(Y_n)) \ge \delta) + \P\big( \ro(X_n) \ne\ro(Y_n)\big)+ \P\big(\dGH(a_nX_n,a_n \ro(X_n)) \ge \delta)\]
for any $\eps>0$. \Cref{Psmall} entails that the first and last terms in the above display may be made arbitrarily small for large~$n$ when~$\eps$ is small enough. For such an $\eps>0$ fixed, using \Cref{Pclose}, we may furthermore assume by the maximal coupling theorem that $(X_n)_{n\in\N}$ is constructed on the same probability space as $(Y_n)_{n\in\N}$ and satisfies
\begin{equation*}\label{eqcoupling}
\lim_{n\to\infty}\P\big(\ro(X_n)=\ro(Y_n)\big)=1\,,
\end{equation*}
so that the middle term may also be made arbitrarily small for large~$n$. Summing up, we can fix an $\eps>0$ such that, for large~$n$, the right-hand side of~\eqref{eqpfthm1} is smaller than~$2\eta$; the result follows.
\end{proof}

\begin{rem}
Alternatively, one could also prove \Cref{mainthm} by first obtaining convergence of the finite dimensional distributions from Propositions~\ref{Pclose} and~\ref{Psmall}, where the latter yields that the restriction contains almost all points and does not distort the distances too much. And then, by proving tightness from that of $\{a_n Y_n:n\in \N\}$, \Cref{Pclose} and \Cref{Psmall}.\ref{PsmallX}.
\end{rem}

We insist on the fact that the above method of proof works in a fairly general sense. More precisely, we inferred the convergence $a_n X_n\to Y$ from $a_n Y_n\to Y$ and the two propositions involving restriction functions. Provided these estimates with adequate restriction functions and the convergence of a reference model of maps, we can conduct the same reasoning. It might also be adaptable to other metrics and objects, not necessarily involving maps.

\section{Comparison of restrictions}\label{seccomp}

In this section, we prove \Cref{Pclose}. From the classical bijective construction of~$Q_{n,p}^\bullet$, we will prove in Section~\ref{secbij} that the restrictions are ``good'' with high probability (\Cref{lem:good}). Assuming this fact, we obtain \Cref{Pclose} by showing that the law of good restrictions in~$X_n$ and in~$Y_n$ are close in total variation distance (\Cref{lem:RND}). The latter fact is obtained from exact counting of quadrangulations.

\subsection{Law of the restrictions}\label{seclawrest}

Fix $\eps>0$ and $n \ge 1$. Let us come back to the definition of the restriction~$\ro$ of a pointed quadrangulation $(\bq,\rho)$ with a simple boundary and its complement. When the procedure works, $\ro(\bq,\rho)$ is a rooted quadrangulation~$\br$ with a simple boundary given along with three distinguished points~$\rho$, $v^-$ and~$v^+$, the last two being on the boundary as in Figure~\ref{fig:constraints}. If~$\br$ is such a map, we say that~$\br$ is an \emph{$(n,\eps)$-restriction map}; we denote by $\ple$ the number of boundary edges between the origin of the root and~$v^-$ in counterclockwise direction and $\pri$ the number of boundary edges between the origin of the root and~$v^+$ in clockwise direction. Finally, let~$\pin$ be the number of boundary edges in between~$v^-$ and~$v^+$ in counterclockwise direction. Since~$\br$ is obtained as a restriction, it is equal to the hull of some ball of some quadrangulation with a simple boundary. The construction of~$\br$ as an $(n,\eps)$-restriction map in particular forces the inequalities
\[
\big\lfloor (\tfrac{1}3-\eps)\, p_{n}\big\rfloor \le \ple \le \lfloor p_{n}/3\rfloor\,,\qquad \pin\ge 1,\qquad \pri\ge 1\,.
\]
If, furthermore, $\br=\ro(\bq,\rho)$ for some $\bq\in\SQ_{n'\!,p'}$, this imposes the additional constraints
\begin{equation}\label{constrrestq}
n'\ge \are{\br}\,,\qquad p'-\pri\ge \lfloor p_{n}/3\rfloor+1\,.
\end{equation}
The first is a basic area constraint, while the second translates the fact that~$v^+$ comes strictly after~$t_{1/3}$. Since $\ple \leq \lfloor p_n/3\rfloor$, the latter implies that the red part of the boundary in Figure~\ref{fig:constraints} has length $p'-\ple-\pri\ge 1$.

\begin{figure}[ht!]
	\centering\includegraphics[width=5.5cm]{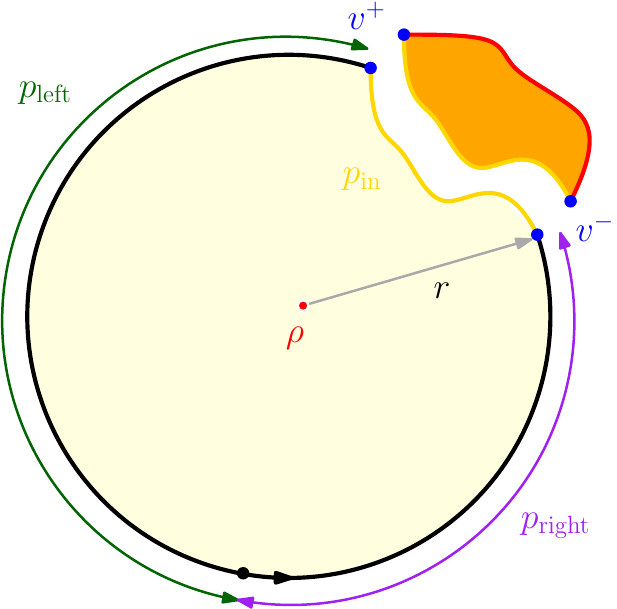}
	\caption{Notation for a restriction map. If the $(n,\eps)$-restriction map~$\br$ appears as the restriction of a quadrangulation with area~$n'$ and perimeter~$p'$, then the constraints~\eqref{constrrestq} must be fulfilled.}
	\label{fig:constraints}
\end{figure}

We denote by
\[
\Sq_{n,p}\de \#\SQ_{n,p}\qquad\text{ for }n\in\N,\, p\in2\N\,,
\]
set $\Sq_{0,2}\de 1$, and set $\Sq_{n,p}\de 0$ otherwise.

\begin{lemma}\label{lemRnpm}
Let~$n'\in\N$, $p'\in2\N$ be such that $p'\ge p_n/2$, and let $\br= (\br, \rho, v^-, v^+)$ be an $(n,\eps)$-restriction map. Then
\[
\P\Big(\ro\big(\tQ_{n'\!,p'}^{\bullet}\big)= \br\Big)=\frac{\Sq_{n'-\are{\br},\,p'-\per{\br}+2\pin}}{(n'+p'/2+1)\,\Sq_{n'\!,p'}} \ind_{\{p'-\pri> p_{n}/3\}}.
\]
\end{lemma}

\begin{proof}
First of all, observe by Euler's characteristic formula that any element of~$\GQ_{n'\!,p'}$ has $n'+p'/2+1$ vertices so the number of pointed quadrangulations with a simple boundary having perimeter~$p'$ and area~$n'$ is the above denominator. 

The result will then follow if we show that the number of maps $\bq \in \SQ_{n'\!,p'}$ such that $\ro(\bq,\rho) = \br$ is equal to the numerator multiplied by the indicator, that is, the number of quadrangulations with a simple boundary having $n'-\are{\br}$ inner faces and perimeter $p'-\per{\br}+2\pin$, that furthermore satisfy~\eqref{constrrestq}. This fact is obtained from a bijection between the set of maps $\bq \in \SQ_{n'\!,p'}$ such that $\ro(\bq,\rho) = \br$ and the set of such quadrangulations with a simple boundary.

More precisely, recalling Figure~\ref{RBall}, observe that a map $\bq \in \SQ_{n'\!,p'}$ such that $\ro(\bq,\rho) = \br$ may be reconstructed from~$\br$ and $\cro(\bq,\rho)$ by identifying the proper parts of the respective boundaries between the vertices~$v^-$ and~$v^+$. Furthermore, choosing as root for instance for $\cro(\bq,\rho)$ the oriented edge directly following~$v^-$ in the contour of the boundary and dropping the two distinguished vertices on the boundary gives a quadrangulation with a simple boundary having $\are{\bq}-\are{\br} = n'-\are{\br}$ inner faces and perimeter $\pin + \per{\bq}-(\per{\br}-\pin)=p'-\per{\br}+2\pin$, and that satisfies~\eqref{constrrestq}. The data of this map together with~$\br$ still allows to reconstruct~$\bq$.

Reciprocally, gluing on the boundary of~$\br$ from~$v^-$ to~$v^+$ any quadrangulation with a simple boundary having area $n'-\are{\br}$ and perimeter $p'-\per{\br}+2\pin$ where~$n'$, $p'$ satisfy~\eqref{constrrestq} gives a pointed quadrangulation with a simple boundary whose restriction is~$\br$. This is because the balls are the same up to the radius where the set~$I$ is reached, and the latter set only depends on~$\br$, not on the glued part. The result follows.
\end{proof}

\subsection{Good restrictions}
For $\delta  \in (0, \eps)$, an $(n,\eps)$-restriction map~$\br$ is called \emph{$(n,\delta)$-good} if
\newlength{\alat}
\settowidth{\alat}{$\ple$}
\begin{alignat}{2}
	\left\lfloor \left(\frac{1}3-\eps\right)\, p_{n}\right\rfloor & \mathbin{\le}{} &\makebox[\alat]{$\ple$}  & \mathbin{\le} \left( \frac{1 }{3}- \delta \right)p_{n}\,,\label{eq:perim1}\\
	 \frac{p_{n}}{2}   & \mathbin{\le}{}  &\pri & \mathbin{\le} \left( \frac{2 }{3}- \delta \right)p_{n}\,,\label{eq:perim2}\\
	\frac{n}{2} & \mathbin{\le}{} &\makebox[\alat]{$\are{\br}$} & \mathbin{\le} (1- \delta)\, n\,, \label{eq:volume}\\
		&& \makebox[\alat]{$\pin$} & \mathbin{\le} \frac{\sqrt{n}}\delta\,.\label{eq:perim3}
\end{alignat}
Note that the inequality involving~$\eps$ always holds for $(n,\eps)$-restriction maps. In words, a restriction is $(n,\delta)$-good if its parameters are in the proper scales: the perimeters~$\pin$, $\ple$ and~$\pri$ are of the same order as~$p_n$ and the volume is of order~$n$. 
See Figure~\ref{good}.

\begin{figure}[ht!]
	\centering\includegraphics[width=7.5cm]{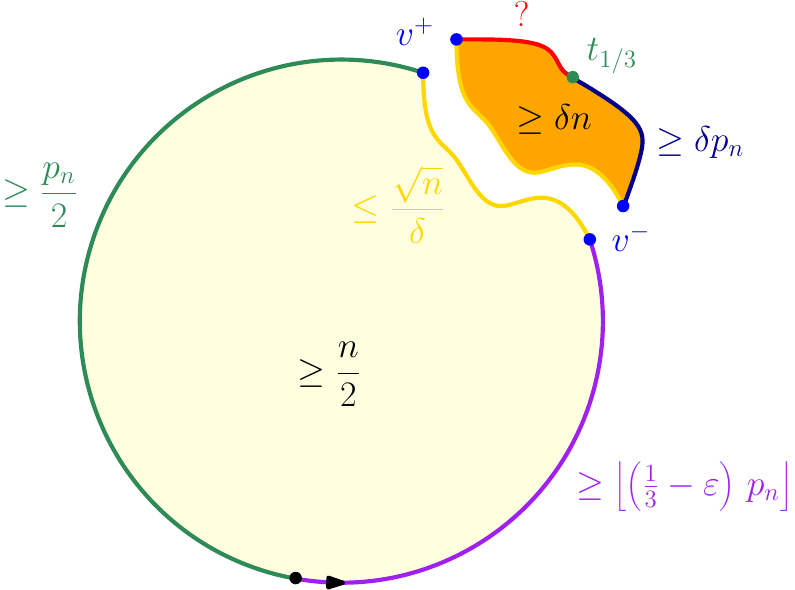}
	\caption{Inequalities defining $(n,\delta)$-good $(n,\eps)$-restriction maps. The length of the red part of the boundary depends on the map of which the restriction map is the restriction. If this map has perimeter~$p_n$, then the length of this red part is between~$\delta p_n$ and~$p_n$.}
	\label{good}
\end{figure}

The following lemma will be proved during the next section from classical bijective constructions. Note that such an estimate is to be expected from usual random maps scaling results.

\begin{lemma}\label{lem:good}
For every $\eta >0$, there exists an arbitrarily small $\eps>0$ and a $\delta\in(0,\eps)$ such that 
\[\liminf_{n\to \infty}\P \Big(\ro\big(\core\big(Q_{n, 3p_n}^{\bullet}\big)\big) \text{ is $(n,\delta)$-good}\Big) \ge 1-\eta. \]
\end{lemma}

The key is then to notice that, if~$\br$ is an $(n,\delta)$-good restriction map, then it may appear as the $(n,\eps)$-restriction of quadrangulations $\bq \in \SQ_{n'\!,p'}$ as soon as $n'\approx n$ and $p' \approx p_{n}$ and that, for such~$n'$, $p'$, the probabilities $\P\big(\ro\big(\tQ^{\bullet}_{n',p'}\big)=\br\big)$ are all very close.

\begin{lemma}\label{lem:RND}
For any $\eta >0$, $\eps>0$ and $\delta >0$, there exist $n_{0}\in\N$ and $\zeta >0$ such that the following holds. For any $n \ge n_{0}$, any $(n,\delta)$-good $(n,\eps)$-restriction map $\br= (\br,\rho,v_{-},v_{+})$ and any $n'\in\N$, $p'\in 2\N$ such that $\big|\frac{n'}{n}-1\big| \le \zeta$ and $\big|\frac{p'}{p_{n}}-1\big| \le \zeta$, we have
\[ \left|\frac{\P\Big(\ro\big(\tQ^{\bullet}_{n,p_n}\big) = \br\Big)}{\P\Big(\ro\big(\tQ^{\bullet}_{n'\!,p'}\big) = \br\Big)}-1 \right| \le \eta\,.\]
\end{lemma}

\begin{proof} 
This relies on \Cref{lemRnpm} and the explicit formula for $\Sq_{m,2\ell}$ found in \cite{BouttierGuitter}:
\begin{align}\label{e.Sq}
	\Sq_{m,2\ell}&= 3^{-\ell}\frac{(3\ell)!}{\ell!(2\ell-1)!}3^m \frac{(2m+\ell-1)!}{(m-\ell+1)!(m+2\ell)!} \qquad\qquad m,\,\ell\ge 1.
\end{align}
Fix $\eta$, $\eps$, $\delta>0$. First, notice that, when~$\zeta$ is sufficiently small, then, for any $(n, \delta)$-good $(n,\eps)$-restriction map~$\br$ and every $n'\in\N$, $p'\in 2\N$ satisfying $\big|\frac{n'}{n}-1\big| \le \zeta$ and $\big|\frac{p'}{p_{n}}-1\big| \le \zeta$, it holds that $n'-\are{\br}\in\N$, $p'-\per{\br}+2\pin\in 2\N$ and $p'-\pri> p_{n}/3$. In this case, by \Cref{lemRnpm},
\begin{equation}\label{eq:ratiobounded}
\frac{\P\Big(\ro\big(\tQ^{\bullet}_{n,p_n}\big) = \br\Big)}{\P\Big(\ro\big(\tQ^{\bullet}_{n'\!,p'}\big) = \br\Big)}
	= \frac{\Sq_{n-\are{\br},\,p_n-\per{\br}+2\pin}}{(n+p_n/2+1)\,\Sq_{n,p_n}} 
	\times \frac{(n'+p'/2+1)\,\Sq_{n'\!,p'}}{\Sq_{n'-\are{\br},\,p'-\per{\br}+2\pin}}\,.
\end{equation}

From~\eqref{e.Sq} and the Stirling formula, we obtain that, for any fixed compact interval $K\subseteq (0,\infty)$, as~$m$, $\ell$ tend to infinity in such a way that $\ell^2/ m \in K$,
\[
\Sq_{m,2 \ell} \sim \frac{\sqrt{3}}{2 \pi}\, 12^m \left(\frac92\right)^\ell m^{-5/2}\, \ell^{1/2} \exp\left(-\frac{9\ell^2}{4m}\right)\,.
\]
Note that the 4 areas and 4 perimeters appearing in the right-hand side of~\eqref{eq:ratiobounded} all tend to infinity from the assumptions on~$n'$, $p'$ and the fact that~$\br$ is $(n, \delta)$-good. Furthermore, for~$\zeta$ small enough and~$n$ large enough, there exist a compact interval $K\subseteq (0,\infty)$ such that each of the 4 corresponding ratios perimeter squared over area belong to~$K$. Using the above equivalent, we deduce that~\eqref{eq:ratiobounded} can be made arbitrarily close to~$1$, provided that~$\zeta$ is small enough and~$n$ large enough.
\end{proof}

We can now gather the above lemmas and \Cref{cvAnPn} in order to prove \Cref{Pclose}.

\begin{proof}[Proof of \Cref{Pclose}]
Recall the notation $X_n= \tQ_{n,p_{n}}^{\bullet}$ and $Y_n= \core\big(Q_{n, 3p_n}^{\bullet}\big)$. Fix $\eta >0$ and find, from Lemmas~\ref{lem:good} and~\ref{lem:RND}, positive numbers~$\eps$, $\delta$, $\zeta>0$ and $n_{0}\in\N$ such that, for $n\ge n_0$, $\ro(Y_n)$ is $(n,\delta)$-good with probability at least $1-\eta$ and the conclusion of \Cref{lem:RND} holds. If $\tilde A_{n} \de \are{Y_n}$ and $\tilde P_n\de\per{Y_n}$, recall that, conditionally on $(\tilde A_{n}, \tilde P_{n})$ and provided that $\tilde A_{n}> n/2$, the core $Y_{n}$ is distributed as a uniform pointed quadrangulation with a simple boundary. 
%
We denote by~$E_n$ the event where both $\big| \frac{\tilde A_{n}}{n} -1\big| \le \zeta$ and $\big| \frac{\tilde P_{n}}{p_n} -1\big| \le \zeta$. From \Cref{lem:RND} with $n'=\tilde A_{n}$ and $p'=\tilde P_{n}$, for any $n \ge n_{0}$ and any $(n,\delta)$-good $(n,\eps)$-restriction map~$\br$,
\begin{multline*}
\left|\P\big(\ro(X_n) = \br\big)-\P\Big(\ro(Y_n) = \br \bigm| \tilde A_{n},\, \tilde P_{n}\Big) \right| 
	\le \eta \,\P\Big(\ro(Y_n) = \br \bigm| \tilde A_{n},\, \tilde P_{n}\Big)\ind_{E_n}+\\
	\left(\P\big(\ro(X_n) = \br\big)+\P\Big(\ro(Y_n) = \br \bigm| \tilde A_{n},\, \tilde P_{n}\Big) \right)\ind_{\bar E_n} \,.
\end{multline*}
As a result,
\begin{equation}\label{dTVgood}
\sum_{ \br\ \text{$(n, \delta)$-good} } \left|\P\big(\ro(X_n) = \br\big)-\P\big(\ro(Y_n) = \br \big) \right|\le \eta+2\P(\bar E_n)\,.
\end{equation}

Increasing $n_0$ if necessary, from \Cref{cvAnPn}, the event~$E_n$ holds for any $n\ge n_0$ with probability at least $1-\eta/2$. In particular, \eqref{dTVgood}, together with the assumption that $\ro(Y_n)$ is $(n,\delta)$-good with probability at least $1-\eta$, yield that, for $n \ge n_{0}$,
\[
\P\big(\ro(X_n)\text{ is $(n,\delta)$-good}\big) \geq \P\big(\ro(Y_n)\text{ is $(n,\delta)$-good}\big) - 2\eta \geq 1- 3\eta\,,
\]
and, finally,
\[
\dTV\big(\ro(X_n),\ro(Y_n)  \big)\le \frac12 \left(2\eta+\P\big(\ro(X_n)\text{ is not $(n,\delta)$-good}\big)+\P\big(\ro(Y_n)\text{ is not $(n,\delta)$-good}\big)\right)\le 3\eta\,.
\]
As a result, 
\[
\liminf_{\eps\to 0}\lsup_{n \to \infty} \dTV\big(\ro(X_n),\ro(Y_n)\big)=0
\]
and we conclude thanks to~\eqref{decreas}, which implies that, for each~$n$, $\dTV\big(\ro(X_n),\ro(Y_n)\big)$ is nonincreasing with~$\eps$.
\end{proof}

\section{Estimates from the bijective construction}\label{secbij}

In this section, we use the classical bijective construction of~$Q_{n, 3p_n}^\bullet$ in order to prove the  rough estimates of \Cref{lem:good}, as well as \Cref{Psmall}. We start with deterministic observations.

\subsection{Bijective encoding by labeled treed bridges}\label{secdefbij}

Let us now recall the classical encoding of quadrangulations with a boundary; this is a particular case of the Bouttier--Di Francesco--Guitter bijection~\cite{BDG04}, which generalizes the famous Cori--Vauquelin--Schaeffer bijection~\cite{CVS81} between plane quadrangulations and so-called well-labeled trees. An encoding object, which we will call a \emph{labeled treed bridge}, consists in:
\begin{itemize}
	\item a rooted cycle $(\rho_0,\rho_1,\dots,\rho_p=\rho_0)$ of length~$p$ for some even $p\in 2\N$, labeled by an integer-valued function~$\lambda$ in such a way that $\lambda(\rho_0)=0$ and $|\lambda(\rho_{i+1})-\lambda(\rho_{i})|=1$ for $0\le i < p$\,;
	\item and, for each $i\in\{0,\dots,p-1\}$ such that $\lambda(\rho_{i+1})=\lambda(\rho_{i})-1$, a plane tree with root vertex~$\rho_i$ whose vertices are labeled by~$\lambda$ in such a way that the labels of any two neighboring vertices differ by~$-1$, $0$, or~$1$.
\end{itemize}
The labels $(\lambda(\rho_0),\lambda(\rho_1),\dots,\lambda(\rho_p))$ of the cycle form a path with $\pm 1$-steps going from~$0$ to~$0$, classically called a \emph{discrete bridge}, so that it has exactly~$p/2$ upsteps and~$p/2$ downsteps. Consequently, a labeled treed bridge built on a cycle of length~$p$ is a forest of~$p/2$ trees (some possibly reduced to the one-vertex tree), labeled by the function~$\lambda$. The \emph{number of edges} in a labeled treed bridge is the sum of the number of edges in its trees. 

Let $m\in \N$ and $p\in 2\N$. The following construction is a bijection between the set of labeled treed bridges with~$p/2$ trees and~$m$ edges and the set of pointed quadrangulations with a boundary having area~$m$ and perimeter~$p$; see Figure~\ref{figbij}. We consider a labeled treed bridges with~$p/2$ trees and~$m$ edges. We first embed, in counterclockwise order, the rooted cycle in the plane, connecting with edges its subsequent elements. We then embed the plane trees inside the cycle, without edge crossings. At this stage, we obtain a map with~$2$ faces, the bounded one of degree $2m+p$ and the unbounded one of degree~$p$.

\begin{figure}[ht!]
	\includegraphics[width=7cm]{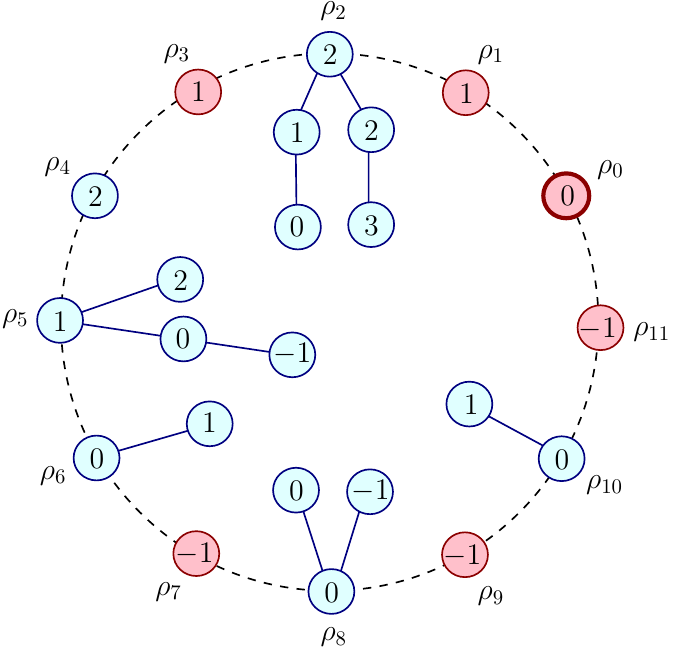}\hfill\includegraphics[width=7cm]{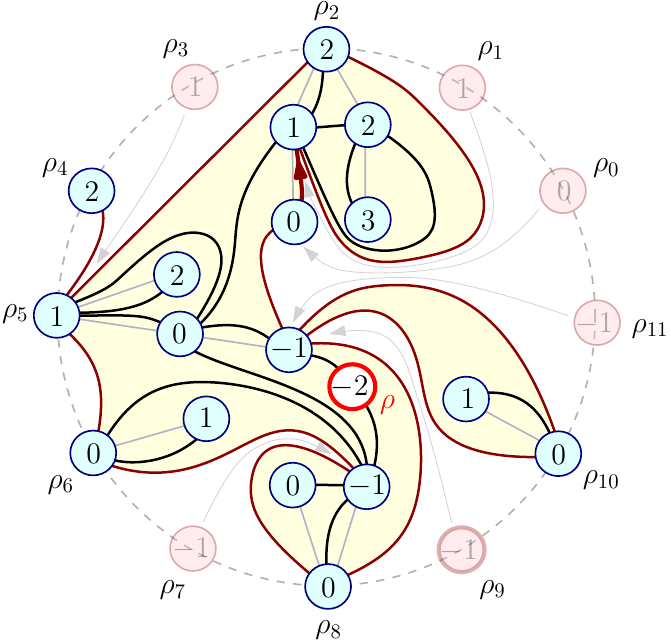}
	\caption{The bijection, from a labeled treed bridge to a pointed quadrangulation with a boundary. On this example, $p=12$, $m=11$, $\lambda_\star=-2$; the tree with root vertex~$\rho_4$ is a one-vertex tree. The edges of the cycle are dashed; its root~$\rho_0$ has a thicker outline. The red vertices precede upsteps in the discrete bridge; they are not vertices of the output map. The gray arrows highlight the correspondence between the cycle and the boundary.}
	\label{figbij}
\end{figure}

Let~$c_0$, $c_1$, \ldots, $c_{2m+p/2-1}$ be the sequence of corners of the bounded face, incident to one of the trees, in contour order, starting from an arbitrary corner. Beware that we ignore the $p/2$ corners incident to the vertices of the cycle that are not the root of a tree. We extend this list by periodicity, setting $c_{2m+p/2+i}=c_i$ for every $i\ge 0$, 
and adding one corner~$c_\infty$ incident to an extra vertex~$\rho$ added inside the bounded face. We extend the definition of~$\lambda$ to corners by letting the label of a corner be equal to the label of the incident vertex. We also set $\lambda_\star\de\min_{i\ge 0} \lambda(c_i) -1$ and $\lambda(c_\infty)=\lambda(\rho)\de\lambda_\star$. Note here again that the minimum is taken over the labels of the tree vertices; the vertices of the cycle without trees are not taken into account. We then define the \emph{successor} of a corner~$c_i$ as the corner $\suc(c_i)\de c_j$ where
\[
j\de\inf\{k>i:\lambda(c_k)=\lambda(c_i)-1\}\in \Z_+\cup\{\infty\}\,.
\]
For each $i\in\{0,\dots,2m+p/2-1\}$, we link by an arc the corner~$c_i$ with its successor, in a non-crossing fashion. We finally discard the original edges. The resulting embedded graph~$\bq$ is a quadrangulation with a boundary pointed at~$\rho$ and rooted as follows. First of all, observe that the original edges of the cycle are in one-to-one correspondence with the oriented edges that are incident to the external face of~$\bq$. Indeed, let us suppose that there is a tree at~$\rho_i$ and the next one is at~$\rho_{i+k}$ for some $k\ge 1$. We denote by~$c_s$ the last corner of the tree with root vertex~$\rho_i$, so that $c_{s+1}$ is the first corner of the tree with root vertex~$\rho_{i+k}$. Then the labels along the cycle in between those trees are $\lambda(\rho_i)$, $\lambda(\rho_i)-1$, $\lambda(\rho_i)$, $\lambda(\rho_i)+1$, $\lambda(\rho_i)+2$, \dots, $\lambda(\rho_i)-2+k=\lambda(\rho_{i+k})$ and the edges linking~$\rho_i$ to~$\rho_{i+k}$ in the cycle correspond to the sequence of~$k$ arcs $c_s\to \suc(c_s)=\suc^{k-1}(c_{s+1})\leftarrow \suc^{k-2}(c_{s+1})\leftarrow \dots \leftarrow \suc(c_{s+1}) \leftarrow c_{s+1}$. The root is then the oriented edge corresponding to the original edge linking~$\rho_{0}$ with~$\rho_1$. See Figure~\ref{figbij}.

In this construction, the edges of the labeled treed bridge are in one-to-one correspondence with the inner faces of the output map~$\bq$ and the vertices of the cycle are in one-to-one correspondence with the corners of the external face of~$\bq$. In the latter correspondence, the labels of corresponding elements are equal. Except from~$\rho$, all the vertices of~$\bq$ are vertices of the labeled treed bridge. Moreover, the labels on~$V(\bq)$ inherited from~$\lambda$ and the convention $\lambda(\rho)=\lambda_\star$ (which we still denote by~$\lambda$) are the relative distances to~$\rho$ in~$\bq$:
\begin{equation}\label{labeldist}
\dd_\bq(v,\rho)=\hat\lambda(v)\de\lambda(v)-\lambda_\star\, ,\qquad v\in V(\bq).
\end{equation}
In the following, the nonnegative integer~$\hat\lambda(v)$ will be called the \emph{shifted label} of~$v$.


\subsection{Reading off information about a restriction from the encoding object}

As in the previous section, we fix $m\in \N$, $p\in 2\N$ and consider a labeled treed bridge with~$p/2$ trees and~$m$ edges and the corresponding pointed quadrangulation with a boundary $\bq^\bullet=(\bq,\rho)$. Using the one-to-one correspondence between the cycle and the boundary of~$\bq$, we let~$\rho_0$, $\rho_1$, \dots, $\rho_{p-1}$ be the corners of the external face of~$\bq$, arranged in contour order, starting from the origin of the root of~$\bq$. For $j\in\{0,1,\dots, p-1\}$, we let~$T(j)$ be the smallest~$k$ such that the corner~$c_k$ is incident to the same vertex as~$\rho_j$. By convention, we also set $T(p)\de 2m+p/2$. 

Remember that the boundary of~$\bq$ is not necessarily simple and that we are interested in its core. We assume that $\core(\bq^\bullet)\neq\wp$ and set $\tilde p\de\per{\core(\bq^\bullet)}$. In contour order, starting from the origin of the root of~$\core(\bq^\bullet)$, we denote by~$\tilde v_0$, $\tilde v_1$, \dots, $\tilde v_{\tilde p-1}$ the vertices of the boundary of~$\core(\bq^\bullet)$. For $i\in\{0,1,\dots,\tilde p-1\}$, we let~$J(i)$ be the smallest~$j$ such that the corner~$\rho_j$ is incident to~$\tilde v_i$.

Beware that we are bound to use 3 timescales: that of the tree corners~$c_0$, \dots, $c_{2m+p/2-1}$, that of the boundary of~$\bq^\bullet$ (given by~$\rho_0$, \dots, $\rho_{p-1}$), and that of the boundary of~$\core(\bq^\bullet)$ (given by~$\tilde v_0$, \dots, $\tilde v_{\tilde p-1}$). We now fix $\eps >0$ and $n \ge 1$ and focus on~$\ro(\core(\bq^\bullet))$; we furthermore assume that~$\bq^\bullet$ is such that this restriction differs from~$\wp$. We use the notation from Section~\ref{secrestr}.

\subsubsection*{Shifted labels of the distinguished vertices}

For the quadrangulation $\core(\bq^\bullet)$, the interval~$I$ from Section~\ref{secrestr} is the set $\big\{\tilde v_i\,,\ \lfloor (1/3-\eps)\, p_{n}\rfloor\le i \le \lfloor p_{n}/3\rfloor\big\}$. From~\eqref{labeldist}, the shifted labels of these vertices are the distances to~$\rho$, so that the minimum~$h$ of these shifted labels is either~$r$ or~$r+1$, where the radius~$r$ is the smallest integer such that $B_{r}(\bq,\rho)$ intersects~$I$. Furthermore, the vertex~$v^-$ is a vertex in~$I$ with shifted label~$r$ or~$r+1$; it is thus a vertex of~$I$ whose shifted label is~$h$ or $h+1$. As the labels between neighboring vertices of the boundary differ by exactly~$1$, the vertex~$v^+$ is between the first boundary vertex after~$t_{1/3}$ with shifted label $h+1$ (included) and the first with shifted label $h-2$ (excluded). We do not need more precision than this; many of these points will become confounded in the scaling limit. For all $0\leq i \leq \tilde p-1$, all the vertices in the noncore component of~$\bq$ attached to~$\tilde v_i$ are farther away from~$\rho$ than~$\tilde v_i$, since any path from~$\rho$ to such a vertex must pass through~$\tilde v_i$. In particular, if a vertex~$\rho_k$ belongs to the noncore component of~$\bq$ attached to~$\tilde v_i$ for some~$i$ and~$k$, it holds that $\hat\lambda(\rho_k)\ge\hat\lambda(\tilde v_i)$. Since $I\subseteq \{\rho_k\,:\ J({\lfloor (1/3-\eps)\, p_{n}\rfloor})\le k \le J({\lfloor p_{n}/3\rfloor})\}$ and any vertex~$\rho_k$ of the latter set is either a~$\tilde v_i$ or belongs to the noncore component of~$\bq$ attached to a~$\tilde v_i$, for some $\lfloor (1/3-\eps)\, p_{n}\rfloor\le i \le \lfloor p_{n}/3\rfloor$, we have
\[
h=\min \big\{\hat\lambda(\rho_k)\,:\ J({\lfloor (1/3-\eps)\, p_{n}\rfloor})\le k \le J({\lfloor p_{n}/3\rfloor}) \big\}\,.
\]

\subsubsection*{Volume estimates}

We let~$i^-$ and~$i^+$ be the indices such that $v^-=\tilde v_{i^-}$ and $v^+=\tilde v_{i^+}$. Note that we thus have $\ple=i^--1$ and $\pri=\tilde p_n-i^+$. We then define the set~$\cS$ of vertices of the trees of the labeled treed bridge whose root vertices belong to $\{\rho_j:\; J({i^-})\le j < J({i^+})\}$ and refer the reader to Figure~\ref{vol}. 

\begin{figure}[ht!]
	\centering\includegraphics[width=13cm]{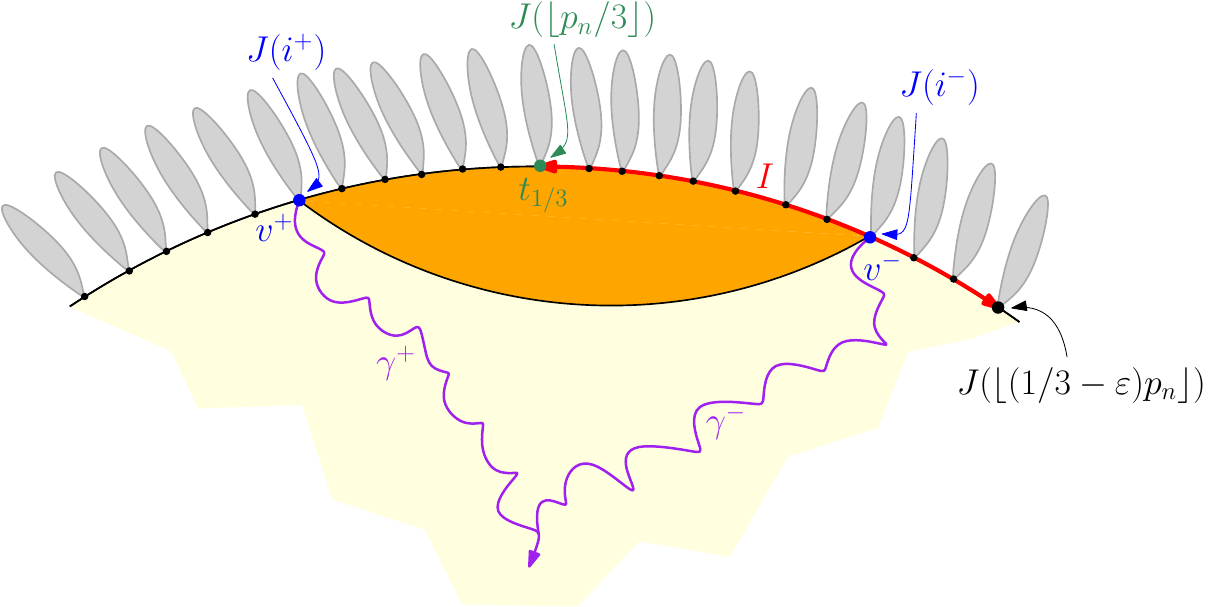}
	\caption{Some notation on the map~$\bq^\bullet$. The noncore components are grayed out. The geodesics~$\gamma^\pm$ link~$v^\pm$ to~$\rho$; they separate $\cro(\core(\bq^\bullet))$ from~$\rho$.}
	\label{vol}
\end{figure}

We claim that all the vertices of~$\cro(\core(\bq^\bullet))$ except at most two belong to~$\cS$. To see this, we let~$\gamma^\pm$ be the leftmost geodesic to~$\rho$ issued from~$c_{T\circ J(i^\pm)}$, that is, the path made of the edges linking~$c_{T\circ J(i^\pm)}$ to its iterate successors. These paths are geodesics thanks to~\eqref{labeldist}. As they start from boundary vertices, they separate $\core(\bq^\bullet)$ into two connected components.

The component that contains~$t_{1/3}$ actually includes $\cro(\core(\bq^\bullet))$. Indeed, first observe that the common boundary between $\ro(\core(\bq^\bullet))$ and $\cro(\core(\bq^\bullet))$ is made of vertices having shifted labels~$r$ or~$r+1$. Since the geodesic~$\gamma^\pm$ visits vertices with decreasing labels, it visits this common boundary at~$v^\pm$ and possibly after its first step only (in the case where $\hat\lambda(v^\pm)=r+1$). We see from the definition of the hull that, if this eventuality occurs, the first edge of~$\gamma^\pm$ is actually part of this common boundary.

Finally, the vertices of the component including $\cro(\core(\bq^\bullet))$ all belong to~$\cS$ or~$\gamma^+$ and, except from possibly the first two, the vertices of~$\gamma^+$ do not belong to~$\cro(\core(\bq^\bullet))$. The claim follows.

\medskip
\begin{figure}[ht!]
	\centering\includegraphics[width=15cm]{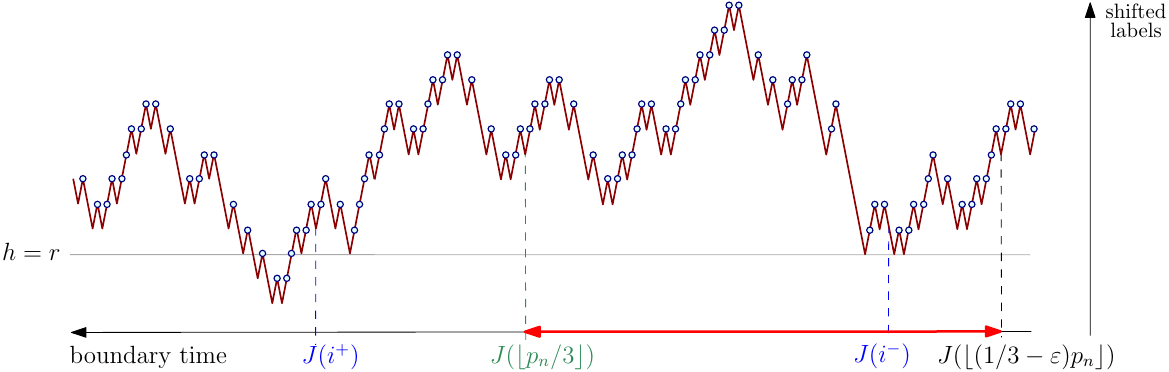}
	\caption{Shifted labels along the boundary. As in Figure~\ref{vol}, the time is that of the boundary of~$\bq^\bullet$ and goes from right to left; and the red double-headed arrow corresponds to the set~$I$. The dark red path represents the shifted labels along the boundary of~$\bq^\bullet$. Before each of its downsteps is a labeled tree; only its root is depicted, with a blue dot as in Figure~\ref{figbij}. On the part of map corresponding to $\big[J(\lfloor (1/3-\eps)\, p_{n}\rfloor), J(\lfloor p_{n}/3\rfloor)\big]$ in the boundary scale, the minimal shifted label along the boundary is~$h$. On this example, $h=r$ and $\hat\lambda(v^-)=\hat\lambda(v^+)=r+1$.}
	\label{bijest1}
\end{figure}

A quadrangulation with area~$n'$ and perimeter~$p'\ge 2$ has $n'+p'/2+1$ vertices, so at least two more vertices than faces; we thus obtain
\[
\are{\cro(\core(\bq^\bullet))}\le\# V(\cro(\core(\bq^\bullet)))-2 \le \#\cS\,.
\]
This upper bound on $\are{\cro(\core(\bq^\bullet))}$ yields a lower bound on $\are{\ro(\core(\bq^\bullet))}$. In order to obtain an upper bound on $\are{\ro(\core(\bq^\bullet))}$, we refer to Figures~\ref{bijest1} and~\ref{bijest2}, and we set 
\[
\cS^{\bge} \de \Big\{ v \in \cS\, :\, \min_{w \in \llbracket \Rt(v),v\rrbracket} \hat{\lambda}(w) \ge r+3 \Big\}\,,
\]
where $\Rt(v)$ denotes the root vertex of the tree\footnote{Recall that the vertices of~$\bq$ different from~$\rho$ are identified with vertices of the encoding labeled treed bridge.} that contains~$v$ and $\llbracket \Rt(v),v\rrbracket$ is the set of vertices on the unique path from~$\Rt(v)$ to~$v$ in the tree (extremities included). We claim that $\cS^{\bge}$ does not intersect~$\ro(\core(\bq^\bullet))$. First, observe from the definitions of~$\core(\bq^\bullet)$ and of~$\ro(\core(\bq^\bullet))$ that the root vertex of a tree in~$\cS$ with shifted label greater than or equal to~$r+3$ does not belong to~$\ro(\core(\bq^\bullet))$, since its label prevents it from being~$v^\pm$. Next, observe that two neighboring vertices in a tree are either linked by an edge of the map if their labels differ or by a path of length two in the map when they have same label. Consequently, any vertex~$v$ can be linked by edges of the map to~$\Rt(v)$ in such a way that the shifted labels on the linking path are all larger than or equal to $\min_{\llbracket \Rt(v),v\rrbracket} \hat\lambda -1$. Such a linking path cannot cross the common boundary between $\ro(\core(\bq^\bullet))$ and $\cro(\core(\bq^\bullet))$ when $\min_{\llbracket \Rt(v),v\rrbracket} \hat \lambda \ge r+3$. The claim follows.

\begin{figure}[ht!]
	\centering\includegraphics[width=15cm]{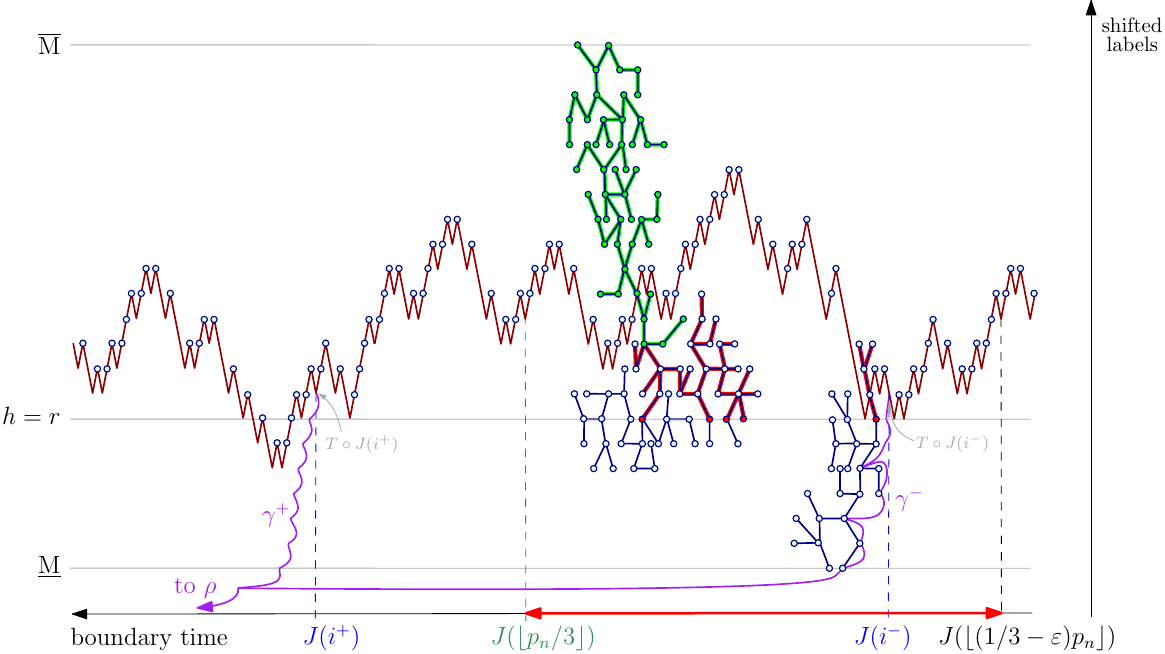}
	\caption{Building on Figure~\ref{bijest1}, we now represented the two trees of the labeled treed cycle that are of most interest: they are properly embedded, with vertices at heights corresponding to their shifted labels. The tree near the middle contains the vertices attaining the maximum~$\overline{\mathrm M}$ over~$\cS$ of the shifted labels and the one to the right contains the vertices attaining the minimum~$\underline{\mathrm M}$ among~$\cS$ of the shifted labels. The geodesic~$\gamma^-$ starts in undisclosed trees until it reaches after three steps the tree with minimal shifted label. The geodesic~$\gamma^+$ entirely lies in undisclosed trees. Highlighted in green are the vertices of~$\cS^{\bge}$ and the edges on their ancestral lines. Highlighted in red are the vertices of~$\cS^{\beq}$ and the edges linking vertices whose ancestral lines are labeled~$r$ or more and that are not highlighted in green.}
	\label{bijest2}
\end{figure}

Combining the above bounds, we have the following estimate for the volume of~$\ro(\core(\bq^\bullet))$:
\begin{equation}\label{estvol}
\are{\core(\bq^\bullet)}-\#\cS \le\are{\ro(\core(\bq^\bullet))}\le\# V(\ro(\core(\bq^\bullet)))\le m+\frac{p}{2}+1-\#\cS^{\bge}\,.
\end{equation}

\subsubsection*{Inner perimeter.}

We will need an upper bound on~$\pin$. Let us introduce
\[
\cS^{\beq} \de \Big\{ v \in \cS\, :\, \hat{\lambda}(v)=r \text{ and } \min_{w\in\llbracket \Rt(v),v\rrbracket\backslash \{v\}} \hat{\lambda}(w) \ge r+1 \Big\}\,.
\]
In words, it is the set of first vertices with label~$r$ when exploring the trees of~$\cS$ from their roots; see Figure~\ref{bijest2}. We claim that this set contains all the vertices with label~$r$ belonging to the common boundary between $\ro(\core(\bq^\bullet))$ and $\cro(\core(\bq^\bullet))$, with the exception of at most one point. Recall that the vertices of $\cro(\core(\bq^\bullet))$ all belong to~$\cS\cup\gamma^+$\,; in particular, the vertices with label~$r$ belonging to the common boundary all lie in~$\cS$ except possibly one (which belongs to~$\gamma^+$), since only one vertex visited by~$\gamma^+$ may have label~$r$. We then follow an argument from \cite[Proposition~18]{CuMeMi13}. Let $v\in \cS\backslash\cS^{\beq}$ be such that~$\hat\lambda(v)=r$. Since $\hat\lambda(\Rt(v))\ge r$ and the successive labels along $\llbracket \Rt(v),v\rrbracket$ differ by at most~$1$, we can find a vertex~$w\in\llbracket \Rt(v),v\rrbracket\backslash \{v\}$ with shifted label~$r$. Considering the geodesics made of the iterate successors of two corners incident to~$w$, one before and one after~$v$ in contour order, we obtain a cycle that separates~$v$ from~$\Rt(v)$. As $\Rt(v)\notin\ro(\core(\bq^\bullet))$ and all the shifted labels along this cycle are smaller than or equal to~$r$, we see that~$v$ cannot be on the common boundary between $\ro(\core(\bq^\bullet))$ and $\cro(\core(\bq^\bullet))$. Adding to this the fact that the shifted labels along the common boundary alternate between~$r$ and~$r+1$, we obtain
\begin{equation}\label{estpin}
\pin \le 2\,(1+ \#\cS^{\beq}) + 1\,.
\end{equation}

\subsubsection*{Distance to the restriction}

Finally, we need a bound on the Gromov--Hausdorff distance\footnote{We cannot use the Hausdorff distance between these sets in the natural embedding of~$\ro(\core(\bq^\bullet))$ within~$\core(\bq^\bullet)$ because this embedding is not isometric: the metric of~$\ro(\core(\bq^\bullet))$ is not the restriction of the metric of~$\core(\bq^\bullet)$. Indeed, between two points of the common boundary, there might exists paths within $\cro(\core(\bq^\bullet))$ that are shorter than a geodesic in~$\ro(\core(\bq^\bullet))$, thus providing ``shortcuts'' in~$\core(\bq^\bullet)$.} between $\core(\bq^\bullet)$ and $\ro(\core(\bq^\bullet))$. Setting
\[
\underline{\mathrm M}\de \min\big\{ \hat\lambda(v)\,:\, v\in \cS\big\}\qquad\text{ and }\qquad
\overline{\mathrm M}\de \max\big\{ \hat\lambda(v)\,:\, v\in \cS\big\}\,,
\]
we claim that
\begin{equation}\label{eq:wena}
\dGH\big(\core(\bq^\bullet),\ro(\core(\bq^\bullet))\big) \le \overline{\mathrm M} - \underline{\mathrm M} + 1.
\end{equation}
We let~$\underline v$ be the merging vertex of~$\gamma^-$ with~$\gamma^+$; it has shifted label $\hat\lambda(\underline v)= \underline{\mathrm M}-1$ and lies in~$\ro(\core(\bq^\bullet))$. Each vertex $v\in \cS$ can then be linked to~$\underline v$ by following the edges linking the iterative successors of any corner incident to~$v$; this results in a path of length smaller than or equal to $\overline{\mathrm M} - \underline{\mathrm M} + 1$. This easily implies that the distortion of the correspondence (see \Cref{apdGH})
\[
\big\{ (v,v)\,:\, v\in \ro(\core(\bq^\bullet))\big\}\cup\big\{ (v,\underline v)\,:\, v\in \cro(\core(\bq^\bullet))\big\}
\]
between $\core(\bq^\bullet)$ and $\ro(\core(\bq^\bullet))$ is less than $2(\overline{\mathrm M} - \underline{\mathrm M} + 1)$. The claim follows.

\subsection{Scaling limits and proofs}

We are interested in the label processes: we set, for $s\in[0,1]$,
\begin{equation}\label{bln}
B(s)\de\lambda\big(\rho_{\lfloor ps\rfloor}\big)\qquad\text{ and }\qquad
L(s)\de\lambda\big(c_{\lfloor (2m+p/2-1)s\rfloor}\big)\,.
\end{equation}
We will also need the so-called \emph{contour process}, defined as follows. For $s\in[0,1]$, if the vertex~$v$ incident to the corner $c_{\lfloor (2m+p/2-1)s\rfloor}$ belongs to the $k$-th tree~$\tt$ of the labeled treed bridge, then
\begin{equation}\label{Cn}
C(s)\de \dd_\tt\big(v,\Rt(v)\big)-k+p/2+1\,.
\end{equation}

We now set $m=n$, $p=3p_n$ and apply our observations to a random quadrangulation $\bq^\bullet=Q^{\bullet}_{n,3p_n}$. To keep track of this, we add a subscript~$n$ in the notation and possibly an~$\eps$ when the quantity depends on the restriction~$\ro$ (as~$i_{n,\eps}^-$ or~$\overline{\mathrm M}_{n,\eps}$ for instance). As the encoding of Section~\ref{secdefbij} is bijective, the labeled treed bridge corresponding to~$Q^{\bullet}_{n,3p_n}$ is uniformly distributed over those with~$3p_n/2$ trees and~$n$ edges. We will need the scaling limit of the random processes of~\eqref{bln} and~\eqref{Cn}, as well as of~$J$ and~$T$, which we now denote by~$B_n$, $L_n$, $C_n$, $J_n$ and~$T_n$ in this probabilistic context. By \cite[Propostion~7 \& Corollary~8]{BetQ}, the following joint convergence holds in distribution, for the uniform topology\footnote{In \cite{BetQ}, the topology considered needs to take into account processes defined on intervals with varying length. It specifies to the uniform topology when working on the fixed interval $[0,1]$.} on the space of bounded functions on~$[0,1]$,
\begin{equation}\label{slbln}
\left(a_n B_n(s),\frac{C_n(s)}{\sqrt{2n}},\frac{T_n\big(\lfloor 3p_{n} s\rfloor\big)}{2n},a_n L_n(s)\right)_{0\le s \le 1}
	\xrightarrow[n\to\infty]{(d)} \big(\fB_s,\fF_s,\cT(s),\fL_s\big)_{0\le s \le 1}\,,
\end{equation}
where~$\fB$ is $3\sqrt{2\alpha}$ times a Brownian bridge on $[0,1]$, $\fF$ is a first-passage Brownian bridge on $[0,1]$ from~$\alpha$ to~$0$, independent of~$\fB$, $\cT$ is the hitting time process\footnote{Recall that this means that $\cT(t) = \inf \{s \ge 0 \,:\, \fF_s=\alpha-t\}$, for $t\in [0,\alpha]$.} associated with~$\fF$, and~$\fL$ is the head of a Brownian snake process built upon~$\fB$ and~$\fF$; we refer to~\cite{BetQ} for the details.

From \cite[Proposition~2.6 \& Lemma~2.7]{GM19cv}, on the event of asymptotically full probability where the core is well defined (\Cref{cvAnPn}), the (simple) boundary of~$\core\big(Q_{n, 3p_n}^{\bullet}\big)$ is ``uniformly spread'' among that of~$Q_{n, 3p_n}^{\bullet}$ in the sense that
\begin{equation}\label{eq:1/3}
\left(\frac{J_n\big({\tilde P_n\wedge \lfloor p_{n} s\rfloor}\big)}{3p_{n}} \right)_{0\le s \le 1} \xrightarrow[n\to\infty]{(d)} \big(s\big)_{0\le s \le 1}\,,
\end{equation}
where $\tilde P_n=\per{\core\big(Q_{n, 3p_n}^{\bullet}\big)}$ as before. 

\begin{proof}[Proof of \Cref{lem:good}]
We fix $\eta >0$. By \Cref{cvAnPn}, the event $\{\tilde P_n\ge p_n/2\}$ holds asymptotically in~$n$ with probability at least $1- \eta/8$. We work on the latter event; in particular, $\core(Q^{\bullet}_{n,3p_{n}}) \neq \wp$. Since the minimum of~$\fB$ over $[1/3- \eps, 1/3]$ is almost surely unique and attained within the open interval $(1/3- \eps, 1/3)$, it follows from~\eqref{slbln} and~\eqref{eq:1/3} that, for~$\eps$ small enough, the event $\{\ro(\core(Q^{\bullet}_{n,3p_{n}})) \neq \wp\}$ holds asymptotically in~$n$ with probability at least $1- \eta/4$ and, on the latter event, \eqref{eq:1/3} holds together with
\begin{equation}\label{Jpm}
\left(\frac{i^{-}_{n, \eps}}{p_{n}}, \frac{i^{+}_{n, \eps}}{p_{n}}
\right) \xrightarrow[n\to\infty]{(d)} \big( \cJ^{-}_{ \eps} , \cJ^{+}_{ \eps} \big)
	\de\Big(\argmin (\fB_{s})_{1/3 - \eps\le s \le 1/3}, \min\big\{ s \ge 1/3 : \fB_{s} = \fB_{ \cJ^{-}_{ \eps}}\big\}\Big).
\end{equation}

From this, we see that we may furthermore find~$\delta\in(0,\eps)$ so that the event where $\ro\big(\core\big(Q_{n, 3p_n}^{\bullet}\big)\big)$ satisfies both~\eqref{eq:perim1} and~\eqref{eq:perim2} holds asymptotically in~$n$ with probability at least $1- 3\eta/8$.

Recall that we now denote by~$\cS_{n,\varepsilon}$ the set previously denoted by~$\cS$, in order to highlight the dependance in~$n$ and $\varepsilon>0$. From~\eqref{slbln}, for every $\eps>0$, the random variable $\#\cS_{n,\eps}/n$ admits a limit in distribution~$\cS_{\eps}$, distributed as $\fF_{\cT(\cJ^{+}_{ \eps})}-\fF_{\cT(\cJ^{-}_{ \eps})}$. From standard properties of Brownian motion, there exists $\tilde \eps>0$ such that, for any $\eps\in(0,\tilde \eps)$, it holds that $\P(\cS_{\eps}\leq 1/3)\geq 1-\eta/8$. Now, taking any $\eps\in(0,\tilde \eps)$, we claim that $\#\cS^{\bge}_{n,\eps}/n$ admits a limit in distribution without atom at~$0$. Taking this claim for granted for a minute and adding \Cref{cvAnPn}, the volume estimate~\eqref{estvol} yields that, up to lowering~$\delta$, the event where $\ro\big(\core\big(Q_{n, 3p_n}^{\bullet}\big)\big)$ satisfies~\eqref{eq:perim1}, \eqref{eq:perim2} and~\eqref{eq:volume} holds asymptotically in~$n$ with probability at least $1- \eta/2$. The latter claim is obtained as follows. First, for each $s\in[0,1]$, we can define a trajectory~$W(s)$ as recording the labels along the ancestral lineage $\llbracket \Rt(v),v\rrbracket$, where we denoted by~$v$ the vertex incident to the corner $c_{\lfloor (2m+p/2-1)s\rfloor}$, as above. The trajectory-valued process~$(W(s))_{0\le s \le 1}$ is the so-called \emph{snake}; in passing, observe that the final value of~$W(s)$ is~$L(s)$, hence the name \emph{head of the snake}. By \cite[Proposition~15]{bettinelli10slr}, the process~$a_n W_n$ actually converges jointly with~\eqref{slbln} toward the so-called \emph{Brownian snake} $(\fW(s))_{0\le s \le 1}$ driven by the process~$\fF$ minus its past infimum, with initial values given by~$\fB$. Then $\#\cS^{\bge}_{n,\eps}/n$ converges in distribution towards
\begin{equation}\label{Sngelim}
\int_{\cJ^{-}_{ \eps}}^{\cJ^{+}_{ \eps}} \ind_{\{\min \fW(s) \ge \min \fB\}} \dd s\,.
\end{equation}
Now, for each $t$ such that $\fF(t)=\min_{0\le s\le t} \fF(s)$, the trajectory $\fW(t)$ is actually the point trajectory $0\mapsto \fL_t$. Since~$\fW$ is a continuous process and~$\fL$ is a.s.\ not identically equal to~$\min \fB$ on $(\cJ^{-}_{\eps}, \cJ^{+}_{\eps})$,  the above integral is almost surely positive.


\medskip	
Finally, for the remaining condition~\eqref{eq:perim3} on the inner perimeter, we see from the estimate~\eqref{estpin} that it is sufficient to prove that $\#\cS_{n, \eps}^{\beq}$ is not large in the scale~$\sqrt{n}$. More precisely, in order to conclude that we can choose~$\delta>0$ small enough so that $\ro\big(\core\big(Q_{n, 3p_n}^{\bullet}\big)\big)$ is $(n,\delta)$-good asymptotically in~$n$ with probability at least $1- \eta$, it is sufficient to show that there exists~$c$ such that $\limsup_n \P(\#\cS_{n, \eps}^\beq\ge c \sqrt{n})\le \eta/2$. This does not follow from the scaling limit results of~\cite{BetQ}; we need to elaborate a bit more.

Recall that, for~$\eps$ small enough, the event $\{\ro(\core(Q^{\bullet}_{n,3p_{n}})) \neq \wp\}$ holds asymptotically in~$n$ with probability at least $1- \eta/4$ and, on the latter event, both~\eqref{eq:1/3} and~\eqref{Jpm} hold. Then the limiting distribution of~$\cJ^{+}_{ \eps}$ ensures that, for $\eps>0$ small enough and~$n$ sufficiently large, the event~$E_n^\eps$ where $\ro(\core(Q^{\bullet}_{n,3p_{n}})) \neq \wp$ and $J_n(i_{n, \eps}^+)\le 2p_n$ occurs with probability at least $1-3\eta/8$. 

From now on, we work on the event~$E_n^\eps$ and restrict our attention to the trees on the bridge between~$0$ and~$2p_n$ (which contain~$\cS_{n, \eps}^{\beq}$). First of all, at the price of a constant, we forget the conditioning on the labeled treed bridge. More precisely, we consider that $(\rho_0,\rho_1,\rho_2,\dots)$ is an infinite sequence labeled by a simple random walk $\lambda(\rho_0)=0$, $\lambda(\rho_1)$, $\lambda(\rho_2)$, \dots, and carrying i.i.d.\ critical Geometric Galton--Watson trees with label differences along edges i.i.d.\ uniformly in $\{-1,0,1\}$ after descending steps. The labeled treed bridge we consider is thus distributed as the $3p_n$ first steps of the later process, conditioned on $\lambda(\rho_{3p_n})=0$ and on having~$n$ edges in the trees. We denote by~$\fS_{n, \eps}^{\beq}$ the set constructed as~$\cS_{n, \eps}^{\beq}$ but with the unconditioned process instead of the labeled treed bridge. Focusing merely on the~$2p_n$ first steps (as we work on~$E_n^\eps$), the Radon-Nikodym derivative of our model with respect to the unconditioned one is explicit (\cite[Lemma~36]{BeMi22}, applied with $a=n$, $k=2p_n$, $l=3p_n$, $\delta=0$) and uniformly bounded by some constant~$C$ (although its inverse is unbounded). This follows by an application of the local limit theorem (\cite[Lemma~37]{BeMi22}) and the fact that the limit of the Radon-Nikodym derivative is bounded (its expression is given in \cite[Equation~(31)]{BeMi22} where $L=3\alpha$ and $L'=2\alpha$). Summarizing, it holds that
\[
\E\big[\#\cS_{n, \eps}^{\beq};\,E_n^\eps\big]\le C\, \E\big[\#\fS_{n, \eps}^{\beq};\,E_n^\eps\big]
\]
Now, for any $ \ell \ge r$, the expected number of first vertices with label~$r$ when exploring from the root such a Galton--Watson tree with root label~$\ell$ is equal to $1$. Indeed, the generating function~$f_{\ell,r}$ for this number is given in~\cite[Equation~(22)]{CM12}: for $x \in [0,1]$
\[
f_{\ell,r}(x) =1 - \frac{2}{\big(\ell - r + a(x)\big) \big(\ell - r + 1 +a(x)\big)}\,,\qquad\text{ where }\quad a(x) = \frac{-1 + \sqrt{1 +8 \left(1 - x \right)^{-1}}}{2}\,,
\]
so that $f_{\ell,r}'(1)=1$. (\textit{To see that this expected number is smaller than or equal to~$1$, one can alternatively consider the first vertices with label~$\ell-1$, then the first vertices with label~$\ell-2$, etc. This makes up a new Galton--Watson tree, whose vertex-set is therefore a subset of the vertex-set of a critical Galton-Watson tree; hence it cannot be supercritical.}) From this, by first conditioning on the discrete bridge, we obtain that $\E\big[\#\fS_{n, \eps}^{\beq};\,E_n^\eps\big]\le 2 p_n$. We conclude by Markov's inequality that 
$\P(\#\cS_{n, \eps}^\beq\ge c \sqrt{n}\,;\,E_n^\eps)\le ({2C p_n})/({c\sqrt n})$, which is asymptotically smaller than~$\eta/8$ for~$c$ large enough.
\end{proof}

\begin{proof}[Proof of \Cref{Psmall}.\ref{PsmallY}]
Recall that $Y_n=\core(Q^{\bullet}_{n,3p_n})$. On the event $\{\ro(Y_n)\neq\wp\}$, we obtain from~\eqref{eq:wena} that
\[
\dGH\big(a_n Y_n, a_n\ro(Y_n)\big)\le a_n \big(\overline{\mathrm M}_{n,\eps} - \underline{\mathrm M}_{n,\eps} + 1\big)\,.
\]
As~$\underline{\mathrm M}_{n,\eps}$ and~$\overline{\mathrm M}_{n,\eps}$ are respectively the minimum and maximum of
\[
\big\{\hat\lambda_n(c_k)\,:\,T_n\circ J_n(i^-_{n,\eps})\le k < T_n\circ J_n(i^+_{n,\eps}) \big\}\,,
\]
by~\eqref{slbln} and~\eqref{eq:1/3}, for~$\eta>0$ fixed and~$\eps$ small enough, the event $\{\ro(Y_n) \neq \wp\}$ holds asymptotically in~$n$ with probability at least $1- \eta$ and, on the latter event, \eqref{Jpm} holds jointly with
\[
a_n \big(\overline{\mathrm M}_{n,\eps} - \underline{\mathrm M}_{n,\eps} + 1\big)\xrightarrow[n\to\infty]{(d)}
	\max_{[\cT({\cJ^{-}_{ \eps}}),\cT({\cJ^{+}_{ \eps}})]} \fL - \min_{[\cT({\cJ^{-}_{ \eps}}),\cT({\cJ^{+}_{ \eps}})]} \fL\,, 
\]
and the latter tends to~$0$ in probability as $ \eps \to 0$. The result follows.
\end{proof}

\subsection{Resampling argument}\label{sec:resampling}

It remains to prove \Cref{Psmall}.\ref{PsmallX}. First of all, note that, for a pointed quadrangulation with a boundary~$\bq^\bullet$, we can use the bijective encoding for~$\bq^\bullet$, for $\core(\bq^\bullet)$ and for $\ro(\core(\bq^\bullet))$. The parts in common of the maps correspond through the encoding bijection to parts in common of the encoding objects. In particular, the labeled treed bridge encoding a map obtained from another by removing some faces can be obtained from the original labeled treed bridge by removing some edges.

We aim at showing that~$X_n$ is close to~$\ro(X_n)$, already knowing that~$Y_n$ is close to~$\ro(Y_n)$ (\Cref{Psmall}.\ref{PsmallY}) and that, after taking a restriction, $X_n$ is close to~$Y_n$ (\Cref{Pclose}). The idea is to apply \Cref{Pclose} with another restriction operation that removes a small part far away from $t_{1/3}$, so that it does not interfere with the local surgery around~$t_{1/3}$. More precisely, we define a second notion of restriction $\rt$ and complement $\crt$ exactly as in Section~\ref{secrestr} except that we reverse the numbering of the boundary vertices, that is, we apply the mapping $0\mapsto 0$ and $i\in\{1,\dots,p-1\}\mapsto p-i$ to the original numbering. See Figure~\ref{resampling}.

\begin{figure}[ht!]
	\centering\includegraphics[width=14cm]{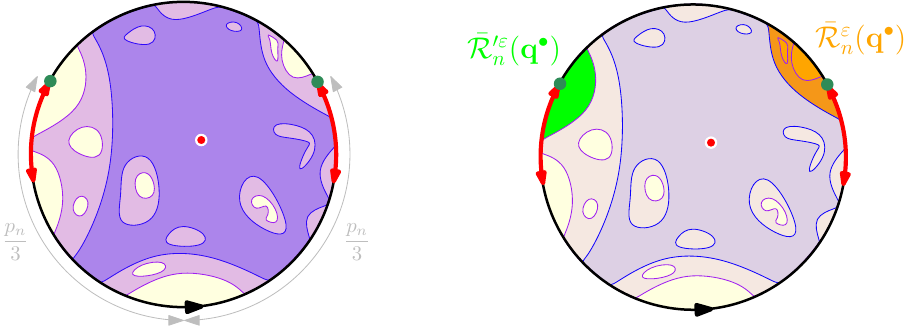}
	\caption{Definition of a second notion of restriction~$\rt$, going ``backwards'' along the boundary. The ball used in the definition of~$\ro(\bq^\bullet)$ is depicted in purple with a blue outline. The (larger) ball involved in the definition of~$\rt(\bq^\bullet)$ is in mauve. The difference between the map~$\bq^\bullet$ and its restriction~$\ro(\bq^\bullet)$ is the same as the difference between~$\rt(\bq^\bullet)$ and~$\ro(\rt(\bq^\bullet))$, provided the complements~$\cro(\bq^\bullet)$ and~$\crt(\bq^\bullet)$ are disjoint.}
	\label{resampling}
\end{figure}

Applying for instance a simple reflection, \Cref{Pclose}, which we have proved by now, also holds for this second notion of restriction: for all $\eps>0$,
\begin{equation}\label{PcloseRt}
\lim_{n \to \infty} \dTV\big(\rt(X_n),\rt(Y_n)  \big)=0\,.
\end{equation}

\begin{proof}[Proof of \Cref{Psmall}.\ref{PsmallX}]
We use~\eqref{eq:wena} as before and highlight the dependence in~$\bq^\bullet$, $n$ and~$\eps$ by denoting the right-hand side bound by~$\ouM(\bq^\bullet,n,\eps)$. Let us start with deterministic considerations and recall how this bound is defined. First, the restriction $\ro(\core(\bq^\bullet))$ defines on the boundary of~$\core(\bq^\bullet)$ and thus on the boundary of~$\bq^\bullet$ the two vertices~$v^+$ and~$v^-$. Then, the part of the boundary of~$\bq^\bullet$ between~$v^+$ and~$v^-$ contains the roots of some trees of the encoding labeled treed bridge of~$\bq^\bullet$. The bound~$\ouM(\bq^\bullet,n,\eps)$ is finally equal to~$1$ plus the difference between the maximal and the minimal label of the vertices that belong to those trees.

If~$\bq^\bullet$ has a simple boundary, then $\bq^\bullet=\core(\bq^\bullet)$, and thus $\dGH\big(\bq^\bullet, \ro(\bq^\bullet)\big) \le \ouM(\bq^\bullet,n,\eps)$. Furthermore, if $\cro(\bq^\bullet)$ and~$\crt(\bq^\bullet)$ are disjoint, then the trees involved in the definitions of $\ouM(\rt(\bq^\bullet),n,\eps)$ and of $\ouM(\bq^\bullet,n,\eps)$ are the same, so that $\ouM(\rt(\bq^\bullet),n,\eps)=\ouM(\bq^\bullet,n,\eps)$. Finally, the vertices considered in the definition of $\ouM(\core(\bq^\bullet),n,\eps)$ form a subset of the vertices involved in the definition of $\ouM(\bq^\bullet,n,\eps)$, so that $\ouM(\core(\bq^\bullet),n,\eps)\le\ouM(\bq^\bullet,n,\eps)$.

We turn to random maps. Since $X_n$ has a simple boundary, we have $\dGH\big(X_n, \ro(X_n)\big) \le \ouM(X_n,n,\eps)$. Now, on the event where $\rt(X_n)=\rt(Y_n)$ and the complements~$\cro(Y_n)$ and~$\crt(Y_n)$ are disjoint,
\[
\ouM(X_n,n,\eps)=\ouM\big(\rt(X_n),n,\eps\big)=\ouM\big(\rt(Y_n),n,\eps\big)=\ouM(Y_n,n,\eps)\le\ouM\big(Q_{n, 3p_n}^{\bullet},n,\eps\big)\,,
\]
the first equality coming from the fact that, on this event, it also holds that~$\cro(X_n)$ and~$\crt(X_n)$ are disjoint. We already showed in the proof of \Cref{Psmall}.\ref{PsmallY} that the latter bound converges, after scaling by~$a_n$, as $n\to\infty$ to a random variable that tends to~$0$ as $\eps\to 0$.

It thus remains to show that the latter event holds asymptotically with probability arbitrarily close to~$1$ for small~$\eps$. Let $\eta>0$. Reasoning as in the proof of \Cref{lem:good}, one can choose~$\eps$ small enough so that~$\cro(Y_n)$ and~$\crt(Y_n)$ are well defined and disjoint with probability at least $1-\eta/2$, asymptotically in~$n$. For such an~$\eps$, by~\eqref{PcloseRt}, for~$n$ large enough, it holds that $\dTV\big(\rt(X_n),\rt(Y_n) \big)<\eta/2$. We conclude thanks to the maximal coupling theorem.
\end{proof}

\appendix
\section{Gromov--Hausdorff topology}\label{apdGH}

Recall that the \emph{Hausdorff distance} between two closed subsets of a metric space $(\cZ, \dd_\cZ)$ is defined as $\dd_{\mathcal H}(A,B) \de \inf \{ \eps>0\,:\, A\subseteq B^\eps \text{ and }B\subseteq A^\eps \}$, where $C^\eps \de \{ x\in \cZ\, :\, \dd_\cZ(x,C)<\eps \}$ denotes the $\eps$-neighborhood of~$C$. The \emph{Gromov--Hausdorff distance} between two compact metric spaces $(\cX,\dd_\cX)$ and $(\cY,\dd_\cY)$ is then defined by
\[
\dGH\big((\cX,\dd_\cX),(\cY,\dd_\cY)\big) \de \inf \Big\{ \dd_{\mathcal H}\big(\varphi(\cX),\psi(\cY)\big)\Big\},
\]
where the infimum is taken over all isometric embeddings $\varphi : \cX \to \cZ$ and $\psi:\cY\to \cZ$ of~$\cX$ and~$\cY$ into the same metric space $(\cZ, \dd_\cZ)$. This defines a metric on the set of isometry classes of compact metric spaces (\cite[Theorem~7.3.30]{burago01cmg}), making it a Polish space\footnote{This is a simple consequence of Gromov's compactness theorem \cite[Theorem~7.4.15]{burago01cmg}.}.

The Gromov--Hausdorff distance may be expressed in terms of correspondences. A \emph{correspondence} between two metric spaces $(\cX,\dd_\cX)$ and $(\cY,\dd_\cY)$ is a subset $\RR\subseteq \cX \times \cY$ such that, for all $x\in \cX$, there is at least one $y\in \cY$ for which $(x,y)\in \RR$ and vice versa. The \emph{distortion} of~$\RR$ is defined as
\[
\dis(\RR) \de \sup \big\{ |\dd_\cX(x,x') - \dd_\cY(y,y')|\,:\, (x,y),(x',y')\in \RR \big\}.
\]
Then, by \cite[Theorem~7.3.25]{burago01cmg}, 
\begin{equation*}
\dGH\big((\cX,\dd_\cX),(\cY,\dd_\cY)\big) = \frac 12 \inf_{\RR} \dis(\RR),
\end{equation*}
where the infimum is taken over all correspondences between $\cX$ and $\cY$.

\begin{acks}[Acknowledgments]
We thank Gr\'egory Miermont for interesting discussions in the elaboration of this work and two anonymous referees for their very detailed comments, which contributed to improve the paper.
\end{acks}

\begin{funding}
The first author, J.B., acknowledges support from Grant ANR-20-CE48-0018 \emph{3DMaps}. The second and third authors, N.C.\ and L.F., acknowledge support from ERC 740943 \emph{GeoBrown}: in particular, this grant supported L.F.\ when he was working at Universit\'e Paris-Saclay, where part of this work was conducted. The fourth author, A.S., was supported by Grant ANID AFB170001, FONDECYT iniciaci\'on de investigaci\'on N${}^\circ$ 11200085 and ERC 101043450 Vortex.
\end{funding}




\end{document}